\documentstyle[amssymb,amsfonts,12pt]{amsart}
\newtheorem{theorem}{Theorem}[section]
\newtheorem{lemma}[theorem]{Lemma}
\newtheorem{corollary}[theorem]{Corollary}
\newtheorem{proposition}[theorem]{Proposition}
\newtheorem{definition}[theorem]{Definition}

\newtheorem{question}[theorem]{Question}

\theoremstyle{remark}
\newtheorem{remark}[theorem]{Remark}

\def\QSet{\mbox{\rm\kern.24em
\vrule width.03em height1.48ex depth-.051ex \kern-.26em Q}}

\def\E{{\mbox{\rm I\kern-.22em E}}}
\def\P{{\bf P}}
\def\D{{\mathcal D}}
\def\T{{\bf T}}

\def\Z{{\bf Z}}
\def\R{{\bf R}}
\def\N{{\bf N}}
\def\C{{\bf C}}
\def\Q{{\bf Q}}
\def\BMO{{\operatorname{BMO}}}
\def\be#1{\begin{equation}\label{#1}}

\def\size{{\operatorname{size}}}

\def\J{{\bf J}}

\def\BMO{{\operatorname{BMO}}}
\def\F{{\mathcal F}}
\def\L{{\mathcal L}}
\def\G{{\mathcal G}}

\def\U{{\mathcal C}}

\def\B{{\mathcal R}}
\def\I{{\mathcal I}}

\def\dist{{\operatorname{dist}}}

\def\bas{\begin{align*}}
\def\eas{\end{align*}}
\def\bi{\begin{itemize}}
\def\ei{\end{itemize}}
\newenvironment{proof}{\noindent {\bf Proof} }{\endprf\par}
\def \endprf{\hfill  {\vrule height6pt width6pt depth0pt}\medskip}
\def\emph#1{{\it #1}}

\begin{document}
\title{On the two dimensional Bilinear Hilbert Transform}

\author{Ciprian Demeter}
\address{School of Mathematics, Institute for Advanced Study, Princeton NJ 08540}
\email{demeter@@math.ias.edu}

\author{Christoph Thiele}
\address{Department of Mathematics, UCLA, Los Angeles CA 90095-1555}
\email{thiele@@math.ucla.edu}

\begin{abstract}
We investigate the Bilinear Hilbert Transform in the plane and the pointwise convergence of bilinear averages in Ergodic theory, arising from $\Z^2$ actions. Our techniques combine novel one and a half dimensional phase-space analysis with more standard one dimensional theory.
\end{abstract}

\maketitle

\tableofcontents

\section{Introduction}

In \cite{LT1}, \cite{LT2}, the following bounds were proved for the one dimensional Bilinear Hilbert Transform
\begin{theorem}
\label{onedimBHT}
\label{dsjfkhwey237r82347jsdkvnfdbvgu02i0-234i9rejgijrtio}
Let $\beta\not\in\{0,1\}$. The bilinear operator defined by the principal value integral
$$H(f,g)(x)=\int f(x+t)g(x+\beta t)\frac{dt}{t}$$
satisfies
$$\|H(f,g)\|_{p_3'}\lesssim \|f\|_{p_1}\|g\|_{p_2},
$$
whenever $\frac{1}{p_1}+\frac{1}{p_2}+\frac{1}{p_3}=1$, $1<p_1,p_2\le \infty$ and $\frac{2}{3}<p_3'<\infty.$
\end{theorem}
In this paper we will investigate two dimensional versions of this result.

More precisely, let $K:\R^2\setminus \{(0,0)\}\to \R$ be a Calder\'on-Zygmund  kernel, that is a kernel satisfying
$$|\partial^{\alpha} \widehat{K}(\xi,\eta)|\lesssim \|(\xi,\eta)\|^{-|\alpha|},$$
for all $\alpha\in\Z_+^2$ with $0\le |\alpha|\le N^4$, and all $(\xi,\eta)\not=(0,0)$. Here $N$ is a large enough positive integer, whose value will not be specified.

We also consider the matrices $A_1,A_2\in M_2(\R)$, and the associated two 
dimensional Bilinear Hilbert Transform

$$T_{A_1,A_2}(F_1,F_2)(x,y):=\int_{R^2}F_1((x,y)+A_1(t,s))F_2((x,y)+A_2(t,s))K(t,s)dtds.$$

We will assume at least one of the $A_i$ is not singular. Due to symmetry, we may and will assume that $A_1$ is not singular.
We will investigate the mapping properties of $T_{A_1,A_2}$ in terms of the spectrum $Spec(B)$ of $B:=A_2A_1^{-1}$.

These questions have parallel interest in Ergodic theory. We investigate the implications of our analysis to Ergodic theory in the last section of the paper.

This material is based upon work supported by the National Science Foundation under agreement No. DMS-0635607. In addition, the first author was supported by NSF Grant DMS-0556389. The second author was supported by NSF Grant DMS-0701302. Any opinions, findings and conclusions or recommendations expressed in this material are those of the authors and do not necessarily reflect the views of the National Science Foundation.

\section{Classification}

We first note that by the change of variables $A_1(t,s)\to (t,s)$ it suffices to analyze operators of the form
$$\int_{R^2}F_1((x,y)+(t,s))F_2((x,y)+B(t,s))K(t,s)dtds.$$
Indeed, since $A_1$ is nonsingular, $\|A_1^{-1}(t,s)\|\sim \|(t,s)\|$, and the kernel
$K(A_1^{-1}(t,s))$ remains Calder\'on-Zygmund.

By dualizing, it suffices to consider instead the associated trilinear forms, defined by
$$\Lambda_{B}(F_1,F_2,F_3):=\int_{R^4}F_1((x,y)+(t,s))F_2((x,y)+B(t,s))F_3(x,y)K(t,s)dxdydtds,$$
and  to understand the range of exponents $p_i$ for which we have\footnote{We restrict attention to the Banach space case $1\le p_i\le \infty$}
$$|\Lambda_{B}(F_1,F_2,F_3)|\lesssim \prod_{i=1}^{3}\|F_i\|_{p_i}.$$
If $B$ is similar to another matrix $C$, say $C=ABA^{-1}$, then $\Lambda_{B}$ and $\Lambda_{C}$ have the same mapping properties. To see this, write
$$\Lambda_{B}(F_1,F_2,F_3)=\int_{R^4}F_1^{A}(A(x,y)+A(t,s))F_2^{A}(A(x,y)+AB(t,s))F_3^{A}(A(x,y))K(t,s)dxdydtds,$$
with $F^A(x,y):=F(A^{-1}(x,y))$. Note that the two functions have similar $L^{p}$ norms.
By changing variables $A(x,y)\to(x,y)$ and $A(t,s)\to(t,s)$ we recover $\Lambda_{C}(F_1^{A},F_2^A,F_3^A)$, and the claim follows.

The forms $\Lambda_{B}$ are associated with multipliers that are singular on the linear subspace (called the singularity) of $\R^6$ determined by the system of equations\footnote{The first two equations describe the support of the multiplier} 
$$\begin{cases} 
\xi_1+\xi_2+\xi_3=0 & \\ \eta_1+\eta_2+\eta_3=0 & \\ \xi_1+b_{11}\xi_2+b_{21}\eta_2=0 & \\  \eta_1+b_{12}\xi_2+b_{22}\eta_2=0& \end{cases}$$
Here $(\xi_i,\eta_i)$ are the frequency variables of $F_i$. The profile of the form $\Lambda_{B}$ depends on the extent to which its singularity is the graph of $(\xi_i,\eta_i)$ over $(\xi_j,\eta_j)$, for $i,j\in\{1,2,3\}$. 
This in turn can fail for one or more pairs $(i,j)$, giving rise to degeneracies. 
The appearance of a hierarchy of degeneracies is the main new phenomenon in two dimensions that we 
address in this paper. It prompts us to use what we think of one and a half dimensional time 
frequency analysis. 
The one dimensional Bilinear Hilbert Transform from 
Theorem \ref{onedimBHT} has only one type of degeneracy, when $\beta=0$ or $\beta=1$. In 
this case the operator is reduced to a linear Hilbert Transform, possibly applied to a product 
of two functions.

We will distinguish the following  cases, in each of which the singularity will be two dimensional. 

\begin{itemize}

\item Case 1. 
$$\{0,1\}\cap Spec(B)=\emptyset.$$
In this case our operator is completely non-degenerate. Its analysis is an adaption of the one dimensional theory to the two 
dimensional context much in the spirit of \cite{pramanikterwilleger}. See Section \ref{sec:nondeg}.

\item Case 2. $$Spec(B)=\{0\}.$$
In this case, by the Jordan canonical form theorem, $B$ will be similar with either $\left[\begin{array}{cc} 0 & 0 \\ 0 & 0\end{array}\right]$ or $\left[\begin{array}{cc} 0 & 1 \\ 0 & 0\end{array}\right]$.

In the first case we get 
$$\Lambda_{B}(F_1,F_2,F_3):=\int_{R^4}F_1(x+t,y+s)(F_2F_3)(x,y)K(t,s)dxdydtds.$$ As in the case of the one dimensional bilinear Hilbert transform, thanks to the full and uniform degeneracy, we immediately conclude that $\Lambda_{B}$ is bounded on $L^{p_1}\times L^{p_2}\times L^{p_3}$ if and only if\footnote{We will ignore the endpoint $L^1$ results} $\frac{1}{p_1}+\frac{1}{p_2}+\frac{1}{p_3}=1$, $1<p_1<\infty$ and $1<p_2,p_3\le\infty$. This follows from the well known two dimensional singular integral theory. 

In the second case, $\Lambda_{B}$ takes the form
$$\Lambda_{B}(F_1,F_2,F_3):=\int_{R^4}F_1(x+t,y+s)F_2(x+s,y)F_3(x,y)K(t,s)dxdydtds.$$
We prove its boundedness in Section \ref{sec:semidegworse}. The singularity can be parametrized as
$$\{(0,a,-a,b,a,-a-b):a,b\in\R\}$$
and one can easily see that neither $(-a,b)$ nor $(a,-a-b)$ is the graph over $(0,a)$.

\item Case 3. $$Spec(B)=\{1\}.$$
This is the case symmetric to Case 2, we will encounter the same possibilities. 
By the Jordan canonical form theorem, $B$ will be similar with either $\left[\begin{array}{cc} 1 & 0\\ 0 & 1\end{array}\right]$ or $\left[\begin{array}{cc} 1 & 1 \\ 0 & 1\end{array}\right]$. The change of coordinates $(x+t,y+s)\to (x,y)$ shows that these subcases correspond to the two subcases of Case 2.

\item Case 4.  $$Spec(B)=\{1,\lambda\},\;\;\lambda\notin\{0,1\}.$$
In this case $B$ is similar to $\left[\begin{array}{cc} \lambda & 0 \\ 0 & 1 \end{array}\right]$. The singularity can be parametrized as
$$\{(a,b,a,-b,-2a,0):a,b\in\R\}$$
and one can easily see that neither $(a,b)$ nor $(a,-b)$ is the graph over $(-2a,0)$. We address this case in detail in Section \ref{sjckldjksecsecsec}. 

\item Case 5.  $$Spec(B)=\{0,\lambda\},\;\;\lambda\notin\{0,1\}.$$
This gives the same possibilities as in Case 4, by the same reason Case 3 and Case 2 are equivalent.

\item Case 6. $$Spec(B)=\{0,1\}.$$
In this case $B$ is similar to $\left[\begin{array}{cc} 1 & 0 \\ 0 & 0\end{array}\right]$, and 
after substituting $y+s$ by $y$ and renaming things, the form is equivalent to
$$\Lambda_{B}(F_1,F_2,F_3):=\int_{R^4}F_1(x+t,y)F_2(x,y+s)F_3(x,y)K(t,s)dxdydtds.$$
The singularity can be parametrized as
$$\{(0,a,b,0,-a,-b):a,b\in\R\},$$
and it can easily be seen that more degeneracies are present here.
The methods we develop in this paper do not seem by themselves sufficient to address this very interesting and 
highly degenerate case. We hope that a further refinement of our techniques will tackle this problem.
\end{itemize}

Due to the degeneracies present in the operators we investigate, the traditional two dimensional\footnote{Here both phase and space are thought of as each representing one dimension} decompositions are ineffective, in that the associated model sums fail to be bounded. 

The main novelty of our approach in this paper lies in the use of one and a half dimensional\footnote{The ambient space for the phase dimension is $\R^2$; decompositions, projections and various structures like tiles, trees etc., will be referred to as one and a half dimensional if they live in $\R^2\times \R$} phase-plane projections. We exemplify this approach in Section \ref{sjckldjksecsecsec} for $B=\left[\begin{array}{cc} \lambda & 0 \\ 0 & 1\end{array}\right]$, and then briefly explain in Section \ref{sec:semidegworse} how our techniques also address the  case $B=\left[\begin{array}{cc} 0 & 1 \\ 0 & 0\end{array}\right]$. 

Finally, we point out the fact that the operators we investigate contain classical one dimensional operators with modulation invariance. We emphasize two important instances.

First, perhaps less surprisingly, the boundedness of the operator analyzed in Section \ref{sub:2} implies the boundedness of the one dimensional Bilinear Hilbert Transform, in some range. To see this, transfer first the result from $\R^2$ to the square torus $\T^2$. Then use $F_1(x,y)=f(x)$ and $F_2(x,y)=g(x)$, $\Phi=\chi_{[-1,1]}$, while $\Psi$ is an appropriate function which decomposes the kernel $1/t$.

Second, and quite strikingly, the boundedness of the operator analyzed in Section \ref{sec:semidegworse} implies the boundedness of the Carleson operator, in some range. In short Carleson's operator is defined by
$$C(f)(y):=\sup_{N\in\R}|\int f(y+s)e^{iNs}\frac{ds}s|.$$
It suffices now to chose $F_1(x,y)=f(y)$, $F_2(x,y)=e^{ixN(y)}g(y)$ and $F_3(x,y)=e^{-ixN(y)}h(y)$ with $\|g\|_{p_2}=\|h\|_{p_3}=1$, $\Psi=\chi_{[-1,1]}$, and $\Phi$ an appropriate function which decomposes the kernel $1/t$, and  to localize the estimates in Section \ref{sec:semidegworse}.

While this may appear as yet another proof of Carleson's classical theorem \cite{Carleson},
the argument of this paper in the special case above reduces largely to the proof
of Carleson's theorem in \cite{LT3}. But the approach in the current paper is further evidence
for a unified proof for bounds of the bilinear Hilbert transform and Carleson's operator, 
following up on the analogy that was stressed in \cite{LT3}.

We refer to the last section for an ergodic theoretic perspective.

\section{The Case 4 and 5} 
\label{sjckldjksecsecsec}
We will analyze the trilinear form associated with $B=\left[\begin{array}{cc} \lambda & 0 \\ 0 & 1\end{array}\right]$, where $\lambda\notin\{0,1\}$. All values of $\lambda\notin\{0,1\}$ are entirely typical, however, to minimize the number of parameters and to ease the exposition we will assume $\lambda=-1$. We thus look at 
 
$$\Lambda(F_1,F_2,F_3)=\int_{\R^4} F_1(x+t,y+s)F_2(x-t,y+s)F_3(x,y)K(t,s)dxdydtds.$$More precisely, we prove
\begin{theorem}
\label{Mainthmforsemidegform}
For each $2<p_i<\infty$ with $\frac{1}{p_1}+\frac{1}{p_2}+\frac{1}{p_3}=1,$ and each $F_i\in L^{p_i}(\R^2)$ we have
$$|\Lambda(F_1,F_2,F_3)|\lesssim \prod_{i=1}^{3}\|F_i\|_{p_i}.$$
\end{theorem}

We remark that we find likely that a more refined analysis
\footnote{In particular, one would have to eliminate some appropriate exceptional sets} 
can push the range of validity of Theorem \ref{Mainthmforsemidegform} to all $1<p_i<\infty$ satisfying $\frac{1}{p_1}+\frac{1}{p_2}+\frac{1}{p_3}=1$. We will not pursue this here.

A simple but important observation shows that
$$\Lambda(F_1,F_2,F_3)=\int_{\R^4} F_1(x+t,y)F_2(x-t,y)F_3(x,y-s)K(t,s)dxdydtds.$$
This formulation of $\Lambda(F_1,F_2,F_3)$ anticipates one of the main features of our approach: we will not do any frequency decomposition for $F_1$ or $F_2$ in the second variable. Indeed, it is not too hard to see that a full two dimensional approach as in the non-degenerate case (see Section \ref{sec:nondeg}), that would amount to a full two dimensional decomposition of all $F_i$, followed by  inserting absolute values on  pieces of the operator associated with each multi-tile, will make the model operator unbounded on all $L^p$ spaces. We omit these details, but we mention that the ``enemy'' here is the fact that the form contains a pointwise product on the $y$ variable of $F_1$ and $F_2$. This pointwise product will not be decomposed any further, but will rather be ignored until the later part of the argument.

\begin{definition}
Let $m:\R^d\to \R$ and let $D\subset \R^d$ be a $d-$ dimensional cube with sidelength $L$. We will say that $m$ is adapted to $D$ of order $M$ if $m$ is supported in $D$ and 
\be{adaptedstrictway}
\|\partial^{\alpha}m\|_{L^{\infty}}\le L^{-|\alpha|}
\end{equation}
for each $\alpha\in \N^d$ with $|\alpha|\le M.$  
\end{definition}

We will often refer to various $m$ as adapted to a certain interval in a more general sense, that is with the understanding that there is an extra (implicit) constant on the right hand side of  \eqref{adaptedstrictway}. This implicit constant will not be stated, but it will always be bounded by a universal constant (i.e. $O(1)$).

In order to discretize $\Lambda(F_1,F_2,F_3)$, we first perform a "cone decomposition" of $K$, that is we decompose smoothly $\widehat{K}$ into pieces localized in (finitely many) cones\footnote{These cones will typically have the same aperture, much smaller than $\pi/2$} centered at the origin (see for example \cite{MTT11}). This decomposition reduces Theorem \ref{Mainthmforsemidegform} to  getting bounds for 
\be{conedecompuuh}
\sum_{k\in\Z}\int_{R^4} F_1(x+t,y)F_2(x-t,y)F_3(x,y-s)\Psi_k(t)\Phi_k(s)dxdydtds,\end{equation}
where $\Psi_k(t)=\frac{1}{2^k}\Psi(t/2^k)$ and $\Phi_k=\frac{1}{2^k}\Phi(s/2^k)$, and $\Psi$ and $\Phi$ are functions whose Fourier transforms are adapted to $[-1/2,1/2]$ of some large order. Moreover, we can assume at least one of $\widehat{\Psi}$ and $\widehat{\Phi}$ is supported away from $0$ and thus $\Psi$ or $\Phi$ has mean zero. This latter condition reflects the fact that the cone to which $\widehat{K}$ is restricted can not intersect both  punctured (frequency) axes. 
To pass from any smooth cone decomposition to cones multipliers that are sume of tensor products
as in (\ref{conedecompuuh}) one can use the standard method of Fourier expansion of pieces
of the cone multiplier.

The bulk of the paper (Section \ref{sub:1}) is devoted to the analysis of the case where $\widehat{K}$ is restricted to a cone that does not intersect the punctured $\eta$ axis\footnote{$\eta$ is the dual of the $y$ variable}. The analysis in the case when the cone does not intersect the punctured $\xi$ axis is somewhat easier\footnote{We will also make the point that similar techniques to the ones we develop to address the first type of cone also apply to the second type of cone}  (at least for exposition purposes), and will be presented in Section \ref{sub:2}.

\subsection{The cone $\int\Phi=0$}\ \
\label{sub:1}

We will thus focus on the case $\int\Phi=0$.

\subsubsection{Discretization}

By using standard reductions, in order to get bounds for \eqref{conedecompuuh} it suffices to prove the boundedness of the model sum
\begin{equation}
\label{model005}
\int_{\R^2}\sum_{Q=\omega_{1}\times\omega_{2}\times {\omega}_{3}\in {\bf Q}}\prod_{i=1}^3\pi_{\omega_i}^{(i)}F_i(x,y)dxdy
\end{equation}
where for $i\in\{1,2\}$, $\pi_{\omega}^{(i)}$ denotes some projection operator (acting on the $x$ variable) associated with a multiplier\footnote{That is $\pi_{\omega}^{(i)}F(x,y)=\int_{\R^2}m_{\omega}(\xi)\widehat{F}(\xi,\eta)e^{i(x\xi+y\eta)}d\xi d\eta$} $m_{\omega}$ adapted to $\omega$ of order $N^4$,
while $\pi_{\omega}^{(3)}$ is the tensor product of a projection as above in the first coordinate and a projection as above on $[|\omega|,2|\omega|]$, in the second coordinate.

Here ${\bf Q}$ is a collection of frequency cubes $Q=\omega_{1}\times\omega_{2}\times {\omega}_{3}$ satisfying the following properties:

\begin{definition}\ \
\label{deffreqint}
\begin{itemize}
\item ${\bf Q}$  is a one parameter family, in that each component $\omega_i$ determines uniquely the other two components of a given $Q$.
\item Each $\omega_i$ is an interval in a fixed shifted dyadic grid\footnote{That is the collection of intervals of the form 
$$\G_{M,j,L}:=\left\{\left[2^{i}\left(l+\frac{L}{M}\right),2^{i}\left(l+\frac{L}{M}+1\right)\right]\;:i\equiv j\pmod {M-1},\;l\in\Z\right\},$$ with $M\ge 3$ an odd integer, $0\le j\le M-2$ and $0\le L\le M-1$} $\D_0$
\item For each $Q$, $|\omega_i|=2^{Jj}$ for some $j\in\Z$, where $J\in \N$ is a fixed large enough  natural number. Such intervals will be referred to as $J$- dyadic.
\item $|\omega_i|=|\omega_i'|$ and $\omega_i\not=\omega_i'$ implies $\dist(\omega_i,\omega_i')\ge 2^{J}|\omega_i|$ 
\item $\omega_1=\omega_2$.
\item There is a (possibly different) shifted dyadic grid $\D_1$ such that for each $\omega_i$, $i\in\{1,2\}$ there exists\footnote{The enlarged intervals $\bar{\omega}_i$ are a technicality needed for the construction of phase space projections for overlapping trees. They are only needed for $i\in\{1,2\}$} $\bar{\omega_i}\in\D_1$ such that $3000\omega_i\subseteq \bar{\omega}_i\subseteq 4000\omega_i$. 
\item $-2\xi\in C_0\omega_3$ whenever $\xi\in\bar\omega_1$, where $C_0$ is some large enough universal constant\footnote{This is achievable since $Q$ is ``close'' to the plane $\xi_1+\xi_2+\xi_3=0$. See for example  \cite{DTT} for details. The precise positioning of each $\omega_3$ with respect to $\omega_1$ is unimportant for our considerations, since it will not affect any type of orthogonality in our argument}.
\end{itemize}
\end{definition}

The properties above are easily achieved by stretching the intervals $\omega_i$ as needed, by an $O(1)$ factor, and by embedding them into intervals (of similar size) of a shifted dyadic grid. The procedure is completely standard, we refer the reader to \cite{DLTT} (see for example section 6) and \cite{DTT} for details. The sparsification induced by the constant $J$ implies that we have to deal with roughly $O(J)$ model sums like that in \eqref{model005}. This is however no problem, since $J=O(1)$.

We anticipate a bit the proof of the boundedness of \eqref{model005}, and mention that the only source of orthogonality will be the fact that $\pi_{\omega_3}^{(3)}$ projects in the second coordinate on intervals of the form $[|\omega_3|,2|\omega_3|]$, which are pairwise disjoint for distinct scales of $\omega_3$. The requirement $\omega_1=\omega_2$ will not generate orthogonality, and in general, we can not do better than that, that is, we can not achieve a separation condition\footnote{This kind of separation condition is achievable in the case of the one dimensional Bilinear Hilbert Transform, and is the main source of orthogonality in that instance} like $\omega_1=C|\omega_2|+\omega_2$. This can be easily seen in the case when $\Psi$  does not have mean zero\footnote{However, if both $\Psi$ and $\Phi$ have mean zero, that is, if the cone does not touch either punctured frequency axis, then this extra separation can be achieved, and the argument gets significantly simpler} (worst case scenario).

We will further discretize \eqref{model005} this time on the spatial side, and for this we introduce some notation.

Let $\eta$ denote a fixed positive function with integral 1 and with Fourier transform supported in $[-2^{-2J}, 2^{2J}]$, satisfying the pointwise estimates 
\begin{equation}\label{eta-bounds}
 C^{-1} (1 + |x|)^{-N^2} \leq \eta(x) \leq C (1 + |x|)^{-N^2},
\end{equation}
for some large enough $C$ that may depend on $N$.

Let $\eta_j$ denote the function $\eta_j(x) := 2^{jJ}\eta(2^{jJ}x)$.  For any subset $E$ of $\R$ or $\R^2$, denote by $\chi_E$ the characteristic function of $E$. If $E\subseteq \R$ we define the smoothed out characteristic function $\chi_{E,j}$ by
$$ \chi_{E,j} := \chi_E * \eta_{j}.$$
For a square $R=I\times J$ we will also use the notation
$$\chi_{R,j}(x,y)=\chi_{I,j}(x)\chi_{J}(y)$$
Note that we smoothen out only in the first coordinate. Note also that 
\be{chi-split}
\chi_{\cup_{\alpha \in A} E_\alpha,j} = \sum_{\alpha \in A} \chi_{E_\alpha,j}.
\end{equation}
whenever $E_\alpha$ are disjoint. 

Note that $\chi_{E,j}$ is a frequency-localized approximation to $\chi_E$.  In fact we have the pointwise estimate
\be{e-diff}
|\chi_{E,j}(x) - \chi_E(x)| \le C (1 + 2^{jJ} \dist(x, \partial E))^{-N^2+1},
\end{equation}
where $\partial E$ is the topological boundary of $E$.

\begin{definition}
A multi-tile $P=R_P\times {Q}_P$ is identified by its spatial component, a $J$- dyadic square $R_P=I_P\times J_P$ from the standard dyadic grid\footnote{Spatial intervals which are referred to as dyadic are always understood to be in the standard dyadic grid}, and by its frequency component, the $J$- dyadic cube ${Q}_P=\omega_{P_1}\times\omega_{P_2}\times {\omega}_{P_3}\in {\bf Q}$, satisfying the property that $|I_P||\omega_{P_1}|=1$. For each such $P$, we denote by $j_P$ the integer such that $|I_P|=2^{-j_PJ}$. We will actually abuse notation and  for each $J$ dyadic square $R$ will denote by $j_R$ the integer such that $R$ has sidelength $2^{-jJ}$. The collection of all multi-tiles is denoted with $\P$.

We will sometimes abuse notation and denote $\omega_{P_1}=\omega_{P_2}$ by $\omega_P$, while the enlarged intervals $\bar{\omega}_{P_1}=\bar{\omega}_{P_2}$ from Definition \ref{deffreqint} by $\bar{\omega}_P$.

If $P$ is a multi-tile, its restrictions $R_P\times\omega_P$ and $R_P\times\bar\omega_P$ will sometimes be referred to as tiles.
\end{definition}

From \eqref{chi-split} we have that 
$$\sum_{Q\in {\bf Q}}\int\int \prod_{i=1}^3\pi_{\omega_i}^{(i)}F_i(x,y)dxdy=\sum_{P\in \P}\int\int \chi_{R_P,j_P}(x,y)\prod_{i=1}^3\pi_{\omega_i}^{(i)}F_i(x,y)dxdy.$$

By incorporating all the reductions made in this section, and by invoking a limiting argument, Theorem \ref{Mainthmforsemidegform} will follow if we prove the following
\begin{theorem}
\label{main2}
Let $\P$ be an arbitrary finite collection of multi-tiles\footnote{We will abuse notation here and use the same letter $\P$ for subcollections}. Then for each $2<p_i<\infty$ we have
\be{mainineq1444}
|\sum_{P\in \P}\int\int \chi_{R_P,j_P}(x,y)\prod_{i=1}^3\pi_{\omega_i}^{(i)}F_i(x,y)dxdy|\lesssim \prod_{i=1}^{3}\|F_i\|_{L^{p_i}}.
\end{equation}
Moreover, the implicit constant only depends on $p_i$
\end{theorem}

The proof of this Theorem is postponed until Section \ref{sec:34}. There are two important ingredients that lie behind this proof: a localized estimate (Proposition \ref{mainsingletreeestimate5}) and a Bessel type inequality (Proposition \ref{selection}).


\subsubsection{Tree selection and sizes}

In this section we organize $\P$ into structured collections called trees.
\begin{definition}
\label{deftreesg66tgsjhu}
Let $\xi\in\R$ and let $R$ be a $J$- dyadic square. We define
$$\omega_{\xi,R}:=[\xi-\frac122^{j_RJ},\xi+\frac122^{j_RJ}]$$
and
$$\bar{\omega}_{\xi,R}:=[\xi-500\times 2^{j_RJ},\xi+500\times 2^{j_RJ}].$$

A tree $(\T,\xi_\T,R_\T)$ is a nonempty collection $\T\subset\P$ of multi-tiles such that for each $P\in\T$ we have $R_P\subseteq R_\T$ and $\bar{\omega}_{\xi_\T,R_\T}\subseteq \bar{\omega}_{P}$. The pair $(\xi_\T,R_\T)$ will be referred to as the top data of the tree. We will write $\omega_\T $ and $\bar{\omega}_\T$ for $\omega_{\xi_\T,R_\T}$ and $\bar{\omega}_{\xi_\T,R_\T}$.

The tree will be referred to as $\textit{lacunary}$ if $\xi_\T\notin 2\omega_P$ for each $P\in\T$ and $non-lacunary$ (sometimes also referred to as $overlapping$) if $\xi_\T\in 2\omega_P$ for each $P\in\T$. 
\end{definition}

\begin{remark}
\label{bothlacandoversingle}
Each multi-tile $P$ gives rise both to an overlapping tree $(\{P\},c(\omega_P)\footnote{Here and in the following, $c(\omega)$ will denote the center of the interval $\omega$},R_P)$ but also to a lacunary tree $(\{P\},\xi, R_P)$, where $\xi$ can be any point in $100\omega_P\setminus 2\omega_P$.
\end{remark}
\begin{remark}
\label{moreinfoforoverlap}
If the tree $(\T,\xi_\T,R_\T)$ is overlapping, we will actually have better localization in terms of the top frequency
$$2\bar{\omega}_\T\subset \bar{\omega}_P$$
for each $P\in\T$. This is a consequence of Definition \ref{deffreqint}.
\end{remark}
\begin{remark}
\label{3lacanyways}
Note that each tree $\T$ can be decomposed as the union of one lacunary tree $\T_l$ and one non-lacunary tree $\T_o$ each of which has the same top data as the original tree. The distinction whether a given tree is lacunary or not will only be made with respect to the first two components (recall that $\omega_P=\omega_{P_1}=\omega_{P_2}$). With respect to the third component, a tree will always have good orthogonality behavior, and can be automatically thought of as 3-lacunary, since
 
$$2\xi_\T\times 0\in C_0(\omega_{P_3}\times [|\omega_{P_3}|,2|\omega_{P_3}|])\setminus 2(\omega_{P_3}\times [|\omega_{P_3}|,2|\omega_{P_3}|]).$$ 
\end{remark}
Let $\J_\T:=\{j_P:\;P\in\T\}$. We will also denote by $j_\T:=j_{R_\T}$.

For each $j\in\J_\T$ we denote by 
$$E_{j,\T}:=\bigcup_{P\in\T:j_P=j}R_P.$$

We remark that due to Definition \ref{deffreqint}, for each $j\in\J_\T$ there is exactly one $Q$ with sidelength $2^{jJ}$ such that $Q$ is the frequency component of a multi-tile $P\in\T$.

For each such $j\in\J_\T$ we define  the spatial cutoffs $\tilde \chi_{j}$, $\tilde{\tilde\chi_{j}}$ and the Fourier cutoff $\tilde \pi_j$ as
\begin{equation}
\label{tchij-def}
\tilde \chi_{j} := \chi_{E_{j,\T},j} = \sum_{P \in \T: j_P=j} \chi_{R_P,j}
\end{equation}
\begin{equation}
\label{hhh}
\tilde{\tilde{\chi_{j}}}:= \sum_{P \in \T: j_P=j} \chi_{I_P,j}(x)\times\chi_{J_P,j}(y)
\end{equation}
and
$$ \tilde \pi_j := \pi_{\omega_i}^{(i)},$$
where $\omega_i$ is the $i^{th}$ component of the unique $\Omega\in\T$ with $|\omega_i|=2^{jJ}$.  

We remark that our notation is sloppy here, the operator $\tilde{\pi}_j$ also
depends on the parameter $i$. We will always write $\tilde{\pi}_j$
in combination with a function $F_i$, and the omitted index is
always the one of the function $F_i$.

\begin{definition}
A tree selection process consists of choosing a tree $\T_1$
from $\P$, then choosing a tree $\T_2$ from $\P\setminus \T_1$
and so on. I.e., at the $k$-th step we choose a tree $\T_k$ from
$\P\setminus (\T_1\cup\dots\cup \T_{k-1})$. We shall refer to the
trees $\T_k$ as the selected trees.
\end{definition}

\begin{definition}\label{max-tree-def}
Consider a subset $\P_0$ of $\P$ and some top data $(\xi,R)$. 
Then the maximal tree $\T^*$ in $\P_0$ with top data $(\xi,R)$
is the set of all $P\in \P_0$ such that $\bar{\omega}_{\xi,R}\subseteq \bar{\omega}_P$ and $R_P\subseteq R$.

A tree selection process is called greedy, if at the $k$-th step
the tree $\T_k$ is maximal in $\P\setminus (\T_1\cup\dots\cup \T_{k-1})$.
\end{definition}

The fact that trees are selected by a greedy selection algorithm will imply regularity, as expressed by Lemma \ref{count}. This in turn will be used repeatedly in the estimates for the phase-space projections in Proposition \ref{phase-space}, in particular they will ensure that various contributions coming from different scales are summable.

We will use the notation 
$$\tilde{\chi}_{I}(x)=(1+\frac{|x-c(I)|}{|I|})^{-1}$$
and 
$$\tilde{\chi}_{R}(x,y)=\tilde{\chi}_{I_R}(x)\tilde{\chi}_{J_R}(y).$$

\begin{definition}
Let $F_i$ be an $L^2$ function    
and let $(\T,\xi_\T,R_\T)$ be a tree.

We first address the case $i\in\{1,2\}$. For each $P\in\T$ we introduce the following notation
$$\|F_i\|_{P}:= \sup_{m_P}\|\tilde{\chi}_{R_P}^{10}(x,y)T_{m_P}(F_i(\cdot,y))(x)\|_{L^2_{x,y}},$$
where  $m_P$ ranges over all functions adapted to $\omega_P$ of order $N^2$.

If the tree is lacunary then we define its $i$-size $\size_i(\T)$ by
$$\size_i(\T):=\left(\frac{1}{|R_\T|}\sum_{P\in\T}\|F_i\|_{P}^2\right)^{1/2}.$$

If the tree is overlapping then we define its $i$-size $\size_i(\T)$ by
$$\size_i(\T):=\sup_{m_\T}\frac{1}{|R_\T|^{1/2}}\|\tilde{\chi}_{R_\T}^{10}(x,y)T_{m_\T}(F_i(\cdot,y))(x)\|_{L^2_{x,y}},$$
where  $m_\T$ ranges over all functions adapted to $10\omega_\T$ of order $N^2$ which also vanish at some point $v_\T\in 10\omega_\T$.

For each $P\in\T$ we also introduce the following notation
$$\|F_3\|_{P,3}:= \sup_{m_P}\|\tilde{\chi}_{R_P}^{10}(x,y)T_{m_P}(F_3(x,y))\|_{L^2_{x,y}},$$
where  $m_P$ ranges over all functions adapted to $\omega_{P_3}\times [|\omega_{P_3}|, 2|\omega_{P_3}|]$ of order $N^2$. 

We now define the 3-size $\size_3(\T)$ by
$$\size_3(\T):=\left(\frac{1}{|R_\T|}\sum_{P\in\T}\|F_3\|_{P,3}^2\right)^{1/2}.$$
\end{definition}

It turns out that controlling the model sum associated with one tree requires a slightly stronger notion of size.

\begin{definition}
Let $\P_0\subseteq \P$ be a finite collection of multi-tiles. 

If $i\in\{1,2\}$ then we define its maximal overlapping $i$-size by
$$\size_i^{o}(\P_0):=\sup_{(\T,\xi_{\T},R_{\T})\atop{\T\subseteq\P_0}}\size_i(\T),$$
where $(\T,\xi_\T,R_\T)$ runs over all overlapping trees with $\T\subseteq\P_0$. Similarly define the  maximal lacunary  $i$-size $\size_i^{l}(\P_0)$ by restricting the supremum to lacunary trees. The maximal $i$-size $\size^{*}_i(\P_0)$ of $\P_0$  is taken to be the largest of $\size_i^{o}(\P_0)$ and $\size_i^{l}(\P_0)$.

Finally, define the maximal $3$-size of $\P_0$ by
$$\size_3^{*}(\P_0):=\sup_{(\T,\xi_{\T},R_{\T})\atop{\T\subseteq\P_0}}\size_3(\T),$$
where $(\T,\xi_{\T},R_{\T})$ runs over all trees (lacunary or overlapping) with $\T\subseteq\P_0$.
\end{definition}

\begin{remark}The size depends on the input function $F_i$, however, to simplify notation we will ignore this dependence. It will always be clear from the context what function is associated with a given size.
\end{remark}

\begin{remark}
Note that the overlapping size controls phase-space projections onto 3 dimensional boxes,  which might in principle be much thinner than a tile. This component of the maximal size is merely a technicality needed to control the norm of the phase-space projection onto an overlapping tree. It will come into the picture through estimate \eqref{unsigned-est-sup}.
\end{remark}
The way we will prove Theorem \ref{main2} is by first proving the following local estimate. 

\begin{proposition}
\label{mainsingletreeestimate5}
Let $(\T,\xi_\T,R_\T)$ be  a tree selected by a greedy algorithm. Let $F_i$ be test functions on $\R^2$ satisfying
\begin{equation}
\label{supbound}
\|F_i\|_{L^{\infty}}\le 1
\end{equation}  
and let $0\le\theta_1,\theta_2<1$ and $0\le \theta_3\le 1.$
Then we have
\begin{equation}
\label{supbound1}
|\sum_{P\in\T}\int_{\R^2}\chi_{R_P,j_P}\prod_{i=1}^3\pi_{\omega_{P_i}}^{(i)}F_i|\lesssim_{\theta_i}|R_\T|\prod_{i=1}^{3}\size_i^{*}(\T)^{\theta_i}.
\end{equation}
\end{proposition}

We postpone the proof of Proposition \ref{mainsingletreeestimate5} to Section \ref{sec:22}.

\subsubsection{The paraproduct estimate}
In this section we prove the following global version of Proposition \ref{mainsingletreeestimate5}.

\begin{proposition}
\label{prop:paraproduct11286}
Let $\xi\in\R$ and let $\Q'\subset\Q$ be a finite collection of frequency cubes $Q=\omega_{1}\times\omega_{2}\times\omega_{3}:=\omega\times\omega\times\omega_{3}$ with the property that $\xi\in\bar\omega$ for each $Q\in \Q'$.
Then
\begin{equation}
\label{supbound167gtre}
|\sum_{Q\in \Q'}\int_{\R^2}\prod_{i=1}^3\pi_{\omega_{i}}^{(i)}F_i|\lesssim \prod_{i=1}^3\|F_i\|_{p_i}
\end{equation}
for each $1< p_i<\infty$ satisfying $\frac{1}{p_1}+\frac{1}{p_2}+\frac{1}{p_3}=1$. Moreover, the implicit constant in \eqref{supbound167gtre} is independent of $\Q'$ and $F_i\in L^{p_i}(\R^2)$, and it only depends on $p_i$.
\end{proposition}
\begin{proof}
Note that since $\xi\in\bar\omega$ the cubes have distinct scales.
For $i\in\{1,2\}$ denote by $m_{\omega_i}$ the multiplier associated with the one dimensional projections $\pi_{\omega_{i}}^{(i)}$.
We split $\Q'$ in two collections. The first collection $\Q_1'$ will consist of those $Q$ for which (at least one of) $m_{\omega_i}$ vanishes at $\xi$. The proof of \eqref{supbound167gtre} immediately follows in this case by estimating
$$|\sum_{Q\in \Q'_1}\prod_{i=1}^3\pi_{\omega_{i}}^{(i)}F_i|$$
by the product of square functions on the $i^{th}$ and third component and a maximal function on the remaining component.

The second collection $\Q_2'$ will consist of those $Q$ for which none of $m_{\omega_i}$ vanishes at $\xi$. It follows that $\xi\in \omega$ for each $Q\in\Q_2' $. Let $\psi$ be a function which equals 1 on the interval centered at $\xi$ with length $1000C_0$ and vanishes outside the double of that interval. We also denote by 
$$\psi_k(\eta)=\frac{1}{2^{k}}\psi(\eta/2^{k}).$$

By writing $m_{\omega_i}=m_{\omega_i}(\xi)\psi_{|\omega|}+(m_{\omega_i}-m_{\omega_i}(\xi)\psi_{|\omega|})$ and by the discussion in the previous case, it follows that it suffices to prove \eqref{supbound167gtre} with $m_{\omega_i}(\xi)\psi_{|\omega|}$ replacing $m_{\omega_i}$. Since $\|m_{\omega_i}\|_{\infty}\lesssim 1$, it further follows that it suffices to prove the following more general estimate
\be{ashhweuiycnsac';iewop'}
|\int_{\R^2}\sum_{k\in \Z}a_k(T_kF_1)(T_kF_2)\pi_{k}F_3|\lesssim \prod_{i=1}^3\|F_i\|_{p_i},
\end{equation}
where $T_k$ is the one dimensional projection associated with $\psi_{kJ}$, $\pi_{k}$ is a two dimensional projection associated with a multiplier adapted to $[2\xi-2^{kJ+10}C_0, 2\xi+2^{kJ+10}C_0]\times [2^{kJ},2^{kJ+1}]$ and $a_k$ is a sequence bounded in absolute value by 1. Moreover, the implicit constant in \eqref{ashhweuiycnsac';iewop'} will only depend on $p_i$. 

To prove \eqref{ashhweuiycnsac';iewop'}, write 
$$T_kF=F-\sum_{k'>k}S_{k'}F,$$
where $S_{k'}:=T_{k'}-T_{k'-1}$.
It remains to control four terms, namely
\be{ashhweuiycnsac';iewop'1}
|\int_{\R^2}\sum_{k\in\Z}a_k(\sum_{k'>k}S_{k'}F_1)F_2\pi_{k}F_3|
\end{equation}
\be{ashhweuiycnsac';iewop'2}
|\int_{\R^2}\sum_{k\in\Z}a_kF_1(\sum_{k'>k}S_{k'}F_2)\pi_{k}F_3|
\end{equation}
\be{ashhweuiycnsac';iewop'3}
|\int_{\R^2}\sum_{k\in\Z}a_kF_1F_2\pi_{k}F_3|
\end{equation}
and 
\be{ashhweuiycnsac';iewop'4}
|\int_{\R^2}\sum_{k\in\Z}a_k(\sum_{k'>k}S_{k'}F_1)(\sum_{k'>k}S_{k'}F_2)\pi_{k}F_3|.
\end{equation}

Let us take a look first at the term in \eqref{ashhweuiycnsac';iewop'4}. 
Due to frequency support  it equals 
$$|\sum_{l_1,l_2\in\{-1,0,1\}}\int_{\R^2}\sum_{k'\in\Z}(S_{k'+l_1}F_1)(S_{k'+l_2}F_2)(\sum_{k<k'}a_k\pi_{k}F_3)|.$$
Each of the nine terms corresponding to various values of $l_1,l_2$ is easily bounded by the product of square functions on the first two functions and the maximal truncation of a two dimensional singular integral on the third function. The estimate then follows from the well known boundedness of these two operators.

The terms \eqref{ashhweuiycnsac';iewop'1} is estimated by the same argument, upon noting that for $k'>k$
$$
|\int_{\R^2}(S_{k'}F_1)F_2\pi_{k}F_3|=|\int_{\R^2}(S_{k'}F_1)(S_{k'+1}F_2+S_{k'}F_2+S_{k'-1}F_2)\pi_{k}F_3|.
$$
A similar argument works for \eqref{ashhweuiycnsac';iewop'2}.

The proof of \eqref{ashhweuiycnsac';iewop'3} is immediate from the boundedness of the two dimensional singular integral operator 
$$T(F_3)=\sum_{k\in\Z}a_k\pi_{k}F_3.$$
\end{proof}

\subsubsection{Proof of Proposition \ref{mainsingletreeestimate5}}
\label{sec:22}
The  proof of Proposition \ref{mainsingletreeestimate5} relies on Proposition \ref{prop:paraproduct11286} and on the considerations in Section \ref{sec:7}, mostly on Proposition \ref{phase-space}. 

By using standard manipulations like in Section 7 from \cite{MTT8}, based on triangle's inequality, \eqref{e-diff}, Lemma \ref{count}, Lemma \ref{triv-bounds} and H\"older's inequality, one can easily reduce Proposition \ref{mainsingletreeestimate5} to proving 

\begin{equation}
\label{supbound177654}
\left|\sum_{j\in\J_\T}\int_{\R^2}(\prod_{i=1}^2\tilde{\chi}_j\tilde{\pi}_jF_i)\tilde{\tilde{\chi}_j}\tilde{\pi}_3F_3\right|\lesssim_{\theta_i}|R_\T|\prod_{i=1}^{3}\size_i^{*}(\T)^{\theta_i}.
\end{equation}

The novelty of \eqref{supbound177654} is that it has the spacial cutoffs attached to each function.
Note that we use both $\tilde{\chi}_j$ and $\tilde{\tilde{\chi}}_j$. There is no reason to smoothen out a spacial component that is not correlated with a frequency localization. The smoothing is used to preserve frequency localization.  

By scale invariance it suffices to assume that $|R_\T|=1$, while by  modulation symmetry we can also assume that the tree sits near the origin, that is $\xi_\T=0$. We may further assume that the tree is either lacunary or overlapping, see Remark \ref{3lacanyways}. These reductions place us in the setting of Section \ref{sec:7} so we have all the results in that section at our disposal.

The proof of \eqref{supbound177654} will follow precisely the same lines as the proof of Proposition 7.1 in \cite{MTT8}. We briefly describe the strategy. One first uses the estimates from Proposition \ref{phase-space} on  how well phase-space projections approximate functions on a tree, (more precisely, \eqref{lacunary13499}, \eqref{lacunary13499ttgdt} and \eqref{error-1}, depending on whether the tree is lacunary or overlapping), to estimate  
$$
|\sum_{j\in\J_\T}\int_{\R^2}(\prod_{i=1}^2\tilde{\chi}_j\tilde{\pi}_jF_i)\tilde{\tilde{\chi}_j}\tilde{\pi}_3F_3|\lesssim_{\theta_i}|\sum_{j\in\J_\T}\int_{\R^2}(\prod_{i=1}^2\tilde{\chi}_j\tilde{\pi}_j\Pi_i(F_i))\tilde{\tilde{\chi}_j}\tilde{\pi}_3\Pi_3(F_3)| +|R_\T|\prod_{i=1}^{3}\size_i^{*}(\T)^{\theta_i},
$$
where $\Pi_i(F_i)$ denotes the phase-space projection of $F_i$ on the tree $\T$.
Then one uses \eqref{fire-1} and \eqref{lacunary13497} to further bound 
$$|\sum_{j\in\J_\T}\int_{\R^2}(\prod_{i=1}^2\tilde{\chi}_j\tilde{\pi}_j\Pi_i(F_i))\tilde{\tilde{\chi}_j}\tilde{\pi}_3\Pi_3(F_3)|$$
by 
$$|R_\T|\prod_{i=1}^{3}\size_i^{*}(\T)^{\theta_i}.$$
We omit the details.

\subsubsection{Deducing Theorem \ref{main2} from Proposition \ref{mainsingletreeestimate5}}
\label{sec:34}
In this section we state a Bessel type inequality that will allow us to deduce Theorem \ref{main2} from Proposition \ref{mainsingletreeestimate5}

The idea is to break  $\P$ into collections of trees $\T$, such that one has control on both the maximal $i$-sizes $\size^*_i(\T)$ and on the $L^1$ norm of the counting function $\sum_\T |R_\T|$. 

The selection of the trees is done by a greedy selection process, which will be
defined in various steps. We need the following definition:

\begin{definition}\label{convex-def}
Call a tree convex, if it is a selected tree in a greedy selection
process. 
Call a subset $\P_0\subseteq \P$ convex, if it is of the  form
$\P\setminus (\T_1\cup \dots\cup \T_k)$ where $\T_1,\dots, \T_k$ are the
selected trees of a greedy selection process.
\end{definition}

\begin{proposition}
\label{selection}  Let $1 \leq i \leq 3$, $\lambda>0$, and suppose that $\P_0$ is a convex collection of multi-tiles such that
\be{m-bound} \size^*_i(\P_0) < 2\lambda.
\end{equation}
Then there exists a collection $\F^{*}$ of pairwise disjoint convex trees in $\P_0$ such that for each $\epsilon>0$ we have
\be{t-size}
 \sum_{\T \in \F^{*}} |R_\T| \lesssim_{\epsilon}\lambda^{-2-\epsilon}\|F_i\|_2^2
\end{equation}
if $i\in\{1,2\}$ and 
\be{t-size3lac}
 \sum_{\T \in \F^{*}} |R_\T| \lesssim \lambda^{-2}\|F_3\|_2^2
\end{equation}
if $i=3$,
and the remainder set $\P' := \P_0\setminus\bigcup_{\T \in \F^{*}} \T$ is convex and satisfies
\be{remainder} 
\size^*_i(\P') < \lambda.
\end{equation}
\end{proposition}

We postpone the proof of this key proposition to the next section.

We continue by noting that by using multilinear interpolation (see \cite{janson}) it suffices to prove Theorem \ref{main2} under the assumption that $F_i=\chi_{E_i}$ are characteristic functions of sets of finite measure.

Starting with $m$ large and working downward, applying Proposition 
\ref{selection} for each $1 \leq i \leq 3$ for each $m$, we obtain

\begin{corollary}\label{select-corollary} Let $\epsilon>0$ be fixed. For every integer $m$ there exists a collection $\F_m$ of pairwise disjoint convex trees in $\P$ such that we have the size estimate
\be{size-est}
\size^*_i(\bigcup_{\T\in\F_m} \T) < (|E_i| 2^{m})^{\frac1{2+\epsilon}}
\end{equation}
for all $1 \leq i \leq 2$ and $m \in \Z$,
\be{size-est3lac}
\size^*_3(\bigcup_{\T\in\F_m} \T) < (|E_i| 2^{m})^{1/2},
\end{equation}
for all $m \in \Z$, 
 the counting function  estimate
\be{tree-count}
\sum_{\T \in \F_m} |R_\T| \lesssim_{\epsilon}  2^{-m}
\end{equation}
for all $m \in \Z$, and the partitioning
\be{partition}
\P = \P_2\cup \bigcup_{m \in \Z} \bigcup_{\T \in \F_m} \T.
\end{equation}
where $\P_2$ is a subset of $\P$ with $\size^*_i(\P_2)=0$ for all $1\le i\le 3$.
\end{corollary}

We also need the following estimate on the maximal size by the Hardy-Littlewood maximal function.

\begin{lemma}
\label{upperboundestforsize}
For all $1\le i\le 3$ and all $F_i\in L^{\infty}(\R^2)$ we have
$$\size_i^{*}(\P)\lesssim \|F_i\|_{\infty}.$$
\end{lemma}
\begin{proof}
It suffices to bound by $C\|F_i\|_{\infty}$ the quantities $\size_i(\T)$ for $i\in\{1,2\}$ and $\T$ either overlapping or lacunary, and $\size_3(\T)$ for general $\T$. The estimate for $i=3$ is entirely classical (see for example the proof of its one dimensional analog, Proposition 6.3 in \cite{MTT8}). The  estimates for $i\in\{1,2\}$ follow by applying a similar argument on fibers above each $y$. Let's take for example a lacunary tree $\T$ (this is the harder case). Denote by $\I:=\{I_P:P\in\T\}$. Note that
\begin{align*}
\size_i(\T)&\le\left(\frac{1}{|R_\T|}\sum_{I\in\I}\int_{\R}\sum_{P\in\T\atop{I_P=I}}\sup_{m_P}\|\tilde{\chi}_{I}^{10}(x)T_{m_P}(F_i(\cdot,y))(x)\|_{L^2_x}^2\tilde{\chi}_{J_P}^{20}(y)dy\right)^{1/2}\\&\lesssim \left(\frac{1}{|J_\T|}\int_{\R}\tilde{\chi}_{J_\T}^{10}(y)\frac{1}{|I_\T|}\sum_{I\in\I}\sup_{m_P}\|\tilde{\chi}_{I}^{10}(x)T_{m_P}(F_i(\cdot,y))(x)\|_{L^2_x}^2dy\right)^{1/2}
\\&\lesssim \left(\frac{1}{|J_\T|}\int_{\R}\tilde{\chi}_{J_\T}^{10}(y)\|F_i(x,y)\|_{L^{\infty}_x}^2dy\right)^{1/2}\\&\lesssim \|F_i\|_{\infty},
\end{align*}
where the penultimate estimate follows from the afore mentioned one dimensional result applied to each function $F(\cdot,y)$. 
\end{proof}

We have now all pieces ready to prove Theorem \ref{main2}.
Choose some $\epsilon<\min\{p_1-2,p_2-2\}$.
In the $i=3$ case we apply Lemma \ref{upperboundestforsize}, \eqref{size-est3lac}, and the fact that $F_3$ is a characteristic function to obtain
\be{size-max}
\size^*_3(\bigcup_{\T\in \F_m}\T) \lesssim \min(2^{m/2} |E_3|^{1/2}, 1)
\end{equation}

Applying \eqref{partition} we may estimate the left-hand side of \eqref{mainineq1444} by
$$
\sum_{m \in \Z} \sum_{\T \in \F_m} |\sum_{P \in \T}
\int\int \chi_{R_P,j_P} \prod_{i=1}^n \pi_{\omega_{P_i}}^{(i)} \chi_{E_i}|.$$
(Observe that the set $\P_2$ gives no contribution, e.g. by
an appropriate application of Proposition \ref{mainsingletreeestimate5}.)

We may apply Proposition \ref{mainsingletreeestimate5} with $\theta_i := \frac{2+\epsilon}{p_i}$ for $1 \leq i \leq 2$ and $\theta_3=1$, and estimate the previous expression by
$$
\lesssim \sum_{m \in \Z} \sum_{\T \in \F_m}
|R_T| (\prod_{i=1}^{2} \size^*_i(\T)^{\frac{2+\epsilon}{p_i}}) \size^*_3(\T).$$
By \eqref{size-est}, \eqref{t-size3lac}, \eqref{size-max} and then \eqref{tree-count}, we may estimate this by
$$
\lesssim \sum_{m \in \Z} 2^{-m}
(\prod_{i=1}^{2} (2^{m} |E_i|)^{1/p_i})
\min(2^{m/2} |E_3|^{1/2},1).$$
By using the fact that $1/p_1+1/p_2+1/p_3=1$ this simplifies to
$$
(\prod_{i=1}^{2} |E_i|^{1/p_i})
\sum_{m \in \Z}
\min(2^{m/2 - m/p_3} |E_3|^{1/2},2^{-m/p_3}).$$
Performing the $m$ summation we obtain the desired estimate
$$
|\sum_{P \in \P}
\int \chi_{R_P,j_P} \prod_{i=1}^3 \pi_{\omega_{P_i}}^{(i)} F_i| \lesssim  
\prod_{i=1}^3 |E_i|^{1/{p_i}}.
$$
and conclude Theorem \ref{main2}.

\subsubsection{The proof of Proposition \ref{selection}}
\label{sec:345}
We start this section by recalling a few results from \cite{DTT}. So far, we have worked with one and a half dimensional trees\footnote{Trees consist of tiles, identified by a spacial component (a square) and a frequency component (an interval). If we think about both space (in our case $\R^2$) and frequency (in our case $\R^2$) as each representing a dimension, a tile becomes a one and a half dimensional object. We adopt the same terminology for a tree}. We will continue to reserve the name "tree" for this particular structure. In addition to this, in the following discussion we will also invoke some results about one dimensional trees.
\begin{definition}
A one dimensional tile $P=I_P\times \omega_P$ is a $J$- dyadic rectangle with unit area. A one dimensional tree $(\T,\xi_\T, I_\T)$ with top data $(\xi_\T, I_\T)$ is a collection of tiles with the property that $\bar\omega_\T\subseteq \bar\omega_P$ and $I_P\subseteq I_\T$ for each $P\in\T$. 

We call $(\T,\xi_\T, I_\T)$  lacunary if for each $P\in\T$ we have $\xi_\T\notin 2\omega_P$.
\end{definition} 

We abuse notation here and use the same notation for one dimensional tiles and multi-tiles, for one dimensional trees and trees, because in our applications one dimensional tiles will arise from multi-tiles while the one dimensional trees will be generated by \emph{reliable} trees. A reliable tree is one with the property that for distinct $P,P'\in\T$ we have that $I_P\not= I_{P'}$ (or equivalently, $I_P\times\omega_P\not= I_{P'}\times\omega_{P'}$). Thus, if $\T$ is a reliable tree, then
$$\{I_P\times \omega_P:P\in\T\},$$
is a one dimensional tree. 
We will refer to it as the  one dimensional tree induced by the tree $(\T,\xi_\T,R_\T)$, and we will assign it the top data $(\xi_\T,I_\T)$.

More generally, a collection of multi-tiles $\P'$ will be called \emph{reliable}, if for distinct $P,P'\in\P'$ we have that $I_P\times\omega_P\not= I_{P'}\times\omega_{P'}$.

\begin{definition}
We say that two lacunary trees $(\T,\xi_\T,R_\T)$, $(\T',\xi_{\T'},R_{\T'})$ are \emph{strongly disjoint} if $\T \cap \T' = \emptyset$, and whenever $P \in \T$, $P' \in \T'$ are such that $\omega_P \subsetneq \omega_{P'}$, then one has $R_\T \cap R_{P'}=\emptyset$, and similarly with $\T$ and $\T'$ reversed.  We define a \emph{forest} to be any collection $\F$ of lacunary trees such that
any two distinct trees $\T, \T'$ in $\F$ are strongly disjoint.
\end{definition}

A similar definition holds for one dimensional trees. Note that the one dimensional tree induced by a reliable lacunary tree is itself lacunary.

For each collection $\F$ of trees we will denote by
$$\|\F\|_{\BMO}:= \sup_R \frac{1}{|R|} \sum_{\T \in \F: R_\T \subseteq R} |R_\T|,$$ where the supremum is taken over all the 
dyadic squares $R$. Define also the counting function 
$$N_\F:=\sum_{\T\in\F}\chi_{R_\T}.$$

The following result from \cite{DTT} shows that in order to achieve $L^1$ control over the counting function of a collection of trees (and this is essentially what we need to prove in Proposition \ref{selection}, see below), we are permitted to lose a logarithmic factor of $\|N_{\F}\|_{\infty}$, as long as the argument also works for all subcollections, and localizes to a $\BMO$ version as well.

\begin{lemma}
\label{cor:7}
 Let $\F$ be a collection\footnote{This lemma is stated in \cite{DTT} with the extra assumption that $\F$ is a forest; however, its proof in \cite{DTT} shows that $\F$ can be an arbitrary collection, actually for all practical purposes, $\F$ can be thought of as merely a collection of dyadic squares $R_\T$. We choose this minimal formulation, but remark that in our application of this lemma, $\F$ will actually be a forest. Along the same lines, we also observe that while the result in \cite{DTT} is stated for one dimensional trees, the extension to our one and a half dimensional setting requires no modifications} of trees such that 
$$\| N_{\F'}\|_1 \le  A\log^2(2+\| N_{\F'}\|_{L^\infty}) \hbox{ and } \| \F'\|_{\BMO} \le B  \log^2(2+\|N_{\F'}\|_{L^\infty}) $$
for all subcollections of trees $\F' \subseteq \F$.  Then for each $\epsilon>0$ we have
$$ \| N_{\F} \|_1 \lesssim_{\epsilon} AB^{\epsilon}$$
where the implicit constant does not depend on $\F$, but only on $\epsilon$.
\end{lemma}

The next result that we recall is a variant of Proposition 13.1. from \cite{DTT} (see also the remark following it). It asserts that the operators $\tilde\chi_{I_P}^{10}T_{m_P}$, where $\P$ ranges through the tiles in a one dimensional forest, are almost orthogonal, with a logarithmic loss in the $L^{\infty}$ norm of the counting function.

\begin{proposition}\label{nmprop}
Let $\F$ be a forest of one dimensional trees $(\T,\xi_\T,I_\T)$. Let $\P':= \bigcup_{\T \in \F}\bigcup_{P:=I_P\times \omega_P\in \T}P$ be the collection of all the one dimensional tiles in the forest. For each  $P \in \P'$  let $m_P$ be a multiplier adapted to $\omega_P$ of order 2. Then
$$\sum_{P \in \P'} \|\tilde\chi_{I_P}^{10}T_{m_P}(f)\|_{2}^2 \lesssim \log^2(2+\|\sum_{\T \in \F} \chi_{I_\T} \|_{L^\infty}) \| f \|_{2}^2,$$
for each $f\in L^2(\R)$.
\end{proposition}

A standard localization argument also gives the following localized variant of Proposition \ref{nmprop}:
\begin{proposition}\label{nmproplocal}
Under the same hypothesis as above, we have for each dyadic interval $I_0$:
$$\sum_{P \in \P':I_P\subseteq I_0} \|\tilde\chi_{I_P}^{10}T_{m_P}(f)\|_{2}^2 \lesssim \log^2(2+\|\sum_{\T \in \F} \chi_{I_\T} \|_{L^\infty}) \| f\tilde\chi_{I_0}^2\|_{L^2}^2.$$
for each $f\in L^2(\R)$.
\end{proposition}

The above results have been proved in \cite{DTT} with the phase-space projections $\tilde\chi_{I_P}^{10}T_{m_P}(f)$ replaced by their variant $\langle f,\phi_s\rangle\phi_s$. The proof of both Propositions \ref{nmprop} and \ref{nmproplocal} runs with no serious modifications. We leave the details to the reader.

We will also need the following consequence of Lemma 10.4 in \cite{DTT}
\begin{lemma}
\label{controlonexpandedcountingfunction}
Let $\B$ be a collection of dyadic squares and $n\ge 0$. Then we can split $\B$ as $\B^{\sharp}\cup \B^{\flat}$ such that
$$\|\sum_{R\in\B^{\sharp}}\chi_{2^nR}\|_{\infty}\lesssim 2^{4n}\|\sum_{R\in\B}\chi_{R}\|_{\infty}^3$$
and 
$$\sum_{R\in\B^{\flat}}|R|\le \frac{1}{2}\sum_{R\in\B}|R|.$$
\end{lemma}

We have now all the tools ready to prove Proposition \ref{selection}.

We first consider the case $i\in\{1,2\}$. Fix such an $i$.

To reduce the maximal size $\size_i^*(\P_0)$ to at most $\lambda$, we need to eliminate all overlapping and lacunary trees $\T$ with $\size_i(\T)$ exceeding  $\lambda$. We will do this in a minimal manner, so that we achieve the desired control over the counting function of the tree tops.

We first take care of the lacunary size. For each $n\ge 0$ and each square $R$ define
$$\tilde\chi_{R,n}:=\tilde\chi_{R}\chi_{2^{n+1}R\setminus 2^nR}.$$
Call a lacunary tree $(\T,\xi_\T,R_\T)$ upper lacunary if $\xi_\T>c(\omega_P)$ for each $P\in\T$ and lower lacunary if $\xi_\T<c(\omega_P)$ for each $P\in\T$.

To guarantee that after the elimination process stops the lacunary size is no greater that $\lambda$, it suffices to make sure that there are no upper or lower lacunary trees $(\T,\xi_\T,R_\T)$ left in the collection, no $n\ge 0$  and no $m_{P}$ which is adapted to $\omega_P$ of order $N^2$ such that

\be{selbadtreelacunary}
\frac{1}{|R_\T|}\sum_{P\in\T}\|\tilde{\chi}_{R_P,n}^{10}(x,y)T_{m_P}(F_i(\cdot,y))(x)\|_{L^2_{x,y}}^2\ge 2^{-n-10}\lambda^2.
\end{equation}

To achieve this, we eliminate lacunary trees according to the following algorithm:
\begin{itemize}
\item Step 0: Set $n=0$, $\P^{*}:=\P_0$, $\F_{bad,n}=\emptyset$.
\item Step 1: Select a "bad" upper lacunary tree $(\T,\xi_\T,R_\T)\in \P^*$, that is a tree satisfying \eqref{selbadtreelacunary}. We also make sure that  $\xi_\T$ is minimal over all trees with this property\footnote{If there are more trees with the same $\xi_\T$ which qualify to be selected at a certain stage, we select any one that maximizes $|R_\T|$}. Put this tree in the collection $\F_{bad,n}$. If no such tree is available, go to Step 4.
\item Step 2: Construct the collection $\U((\T,\xi_\T,R_\T),n)$ to consist of the following convex trees: for each $J$- dyadic square $R\subset 2^{n+3}R_\T$ with sidelength equal to that of $R_\T$ (and there are $2^{2n+6}$ of them) we let $\T_R$ be the maximal tree with top data $(\xi_\T,R)$; the collection $\U((\T,\xi_\T,R_\T),n)$ will consist of all these (at most $2^{2n+6}$ trees). Eliminate all these convex trees from $\P^{*}$, that is, reset $\P^{*}:=\P^{*}\setminus \bigcup_{\T'\in \U((\T,\xi_\T,R_\T),n)}\T'$
\item Step 3: Go to Step 1
\item Step 4: Reset $n:=n+1$, $\F_{bad,n}=\emptyset$. and go to Step 1. 
\end{itemize}
While the algorithm runs forever, it will produce no bad trees for large enough $n$, since $\P_0$ is finite. After we are done with eliminating the upper trees, if $\P^{*}$ is nonempty we repeat the algorithm for lower trees (the only difference is that in Step 1, $\xi_\T$ will be maximal). 

It suffices to prove that for each $n\ge 0$ and each $\epsilon>0$
\be{besselforlac8}
\sum_{\T\in\F_{bad,n}}|R_\T|\lesssim_{\epsilon} 2^{-3n}\lambda^{-2-\epsilon}\|F_i\|_2^2.
\end{equation}

We will prove this by using the results about one dimensional trees from the beginning of this section.

By Lemma \ref{cor:7}, to achieve \eqref{besselforlac8}, it suffices for each $\F'\subseteq \F_{bad,n}$ to prove the following
\be{besselforlac8subfoest}
\sum_{\T\in\F'}|R_\T|\lesssim  2^{-3n}\lambda^{-2}\|F_i\|_2^2\log^2(2+\|N_{\F'}\|_{L^{\infty}}),
\end{equation}
and 
\be{besselforlac8subfoestloc}
\frac{1}{|R|}\sum_{\T\in\F':R_\T\subset R}|R_\T|\lesssim  2^{-3n}\lambda^{-2}\log^2(2+\|N_{\F'}\|_{L^{\infty}})
\end{equation}
for each dyadic $R$.

We only prove \eqref{besselforlac8subfoest}, then \eqref{besselforlac8subfoestloc} will follow by localizing the techniques (In particular, appealing to Proposition \ref{nmproplocal} rather than Proposition \ref{nmprop}).

We apply Lemma \ref{controlonexpandedcountingfunction} to the collection $\B:=\{R_\T:\T\in\F'\}$ and to $n+1$, and denote by ${\F'}^{\sharp}$ and ${\F'}^{\flat}$ the two collections of trees that arise by this application. It suffices to prove
\be{besselforlac8subfoestelab}
\sum_{\T\in{\F'}^{\sharp}}|R_\T|\lesssim  2^{-3n}\lambda^{-2}\|F_i\|_2^2\log^2(2+\|N_{\F'}\|_{L^{\infty}}).
\end{equation}

Denote $N':=\|N_{\F'}\|_{L^{\infty}}$. We certainly have
\be{jhcdlhUDYEUUIDLUIW}
\|\sum_{\T\in{\F'}^{\sharp}}\chi_{2^{n+1}R_\T}\|_{\infty}\lesssim 2^{4n}{N'}^3.
\end{equation}

For each $y\in\R$ and each $\T\in{\F'}^{\sharp} $ denote by $\T_y$ the subtree consisting of all $P\in\T$ such that $y\in 2^{n+1}J_P$. Split the collection of multi-tiles $\bigcup_{\T\in{\F'}^{\sharp}}\T_y$ into at most $2^{n+2}$ reliable subcollections. Denote them with $\U_{i,y}$, $i\in I_y$ with  $\#I_y\le 2^{n+2}$. Thus, each $\U_{i,y}$ will consist of the union of reliable trees, each of which is a subtree of one of the trees\footnote{Each $\T\in{\F'}^{\sharp} $ will provide at most one such subtree for each given $\U_{i,y}$} $\T_y$, with $\T\in{\F'}^{\sharp} $.

We claim that for each given  $y$ and $i\in I_{y}$, the collection of the one dimensional trees induced by the trees in $\U_{i,y}$ will form a (one dimensional) forest. 

To see this, note first that the induced one dimensional trees are lacunary, since the subtree of a lacunary tree is itself lacunary.

Let now $(\T_{rel}, \xi_\T,R_\T)$ and $(\T'_{rel},\xi_{\T'},R_{\T'})$ be the subtrees of $(\T,\xi_\T,R_\T)$ and $(\T',\xi_{\T'},R_{\T'})$ that are in $\U_{i,y}$. The fact that their induced one dimensional trees are disjoint (as collections of tiles) follows from the fact that $\U_{i,y}$ is reliable.

The proof of the fact that  the induced one dimensional trees $(\T_{rel}, \xi_\T,I_\T)$ and $(\T'_{rel},\xi_{\T'},I_{\T'})$  are strongly disjoint goes by contradiction. Assume $P \in \T_{rel}$, $P' \in \T'_{rel}$ are such that $\omega_P \subsetneq \omega_{P'}$ and $I_\T \cap I_{P'}=\emptyset$. The first condition, the upper lacunarity of both $(\T,\xi_\T,R_\T)$ and $(\T',\xi_{\T'},R_{\T'})$, and the fact that $|\omega_P|\le 2^{-J}|\omega_{P'}|$ easily implies that the tree $\T$ was selected before $\T'$. It also follows that $\bar\omega_\T\subset \bar\omega_{P'}$. On the other hand, $I_\T \cap I_{P'}=\emptyset$ together with the fact that $y\in 2^{n+1}J_\T\cap 2^{n+1}J_{P'}$ implies that $P'$ would have qualified to be eliminated when $\T$ was eliminated, that is before the selection of $\T'$ (more precisely, $P'$ is in one of the trees in $\U((\T,\xi_T,R_\T),n)$). The contradiction is immediate.

Since for each $y$ and $i\in I_y$, each $\T\in {\F'}^{\sharp}$ contributes with at most one subtree to $\U_{i,y}$, and since $y\in 2^{n+1}J_\T$ for each such $\T$, it follows by \eqref{jhcdlhUDYEUUIDLUIW} that the counting function for the one dimensional forest induced by $\U_{i,y}$ obeys the bound\footnote{This holds for a.e. $y$}
\be{jhcdlhUDYEUUIDLUIWonedim}
\|\sum_{\T\in{\F'}^{\sharp}\atop{\T\;contributes\;to\;\U_{i,y}}}\chi_{I_\T}(x)\|_{L^\infty(x)}\lesssim 2^{4n}{N'}^3.
\end{equation}

By \eqref{selbadtreelacunary} we have
\begin{align*}
\sum_{\T\in{\F'}^{\sharp}}|R_\T|&\lesssim 2^{-19n}\lambda^{-2}\int_{y\in\R}\sum_{\T\in {\F'}^{\sharp}}\sum_{P\in\T:y\in 2^{n+1}J_P}\|\tilde{\chi}_{I_P}^{10}(x,y)T_{m_P}(F_i(\cdot,y))(x)\|_{L^2_{x}}^2dy\\&= 2^{-19n}\lambda^{-2}\int_{y\in\R}\sum_{\T\in {\F'}^{\sharp}}\sum_{P\in\T_y}\|\tilde{\chi}_{I_P}^{10}(x,y)T_{m_P}(F_i(\cdot,y))(x)\|_{L^2_{x}}^2dy\\&= 2^{-19n}\lambda^{-2}\int_{y\in\R}\sum_{i\in I_y}\sum_{P\in \U_{i,y}}\|\tilde{\chi}_{I_P}^{10}(x,y)T_{m_P}(F_i(\cdot,y))(x)\|_{L^2_{x}}^2dy
\end{align*}

Using the fact that for each $y$ and $i$, $\U_{i,y}$ is a one dimensional forest with counting function estimate like in \eqref{jhcdlhUDYEUUIDLUIWonedim}, it 
follows by Proposition \ref{nmprop} that the expression above can further be bounded by 
$$
2^{-18n}\lambda^{-2}\int_{y\in\R}\|F_i(x,y)\|_{L^2_{x}}^2\log^2(2+2^{4n}{N'}^3)dy\lesssim 2^{-17n}\lambda^{-2}\|F_i(x,y)\|_{L^2_{x,y}}^2\log^2(2+N').
$$
This ends the proof of \eqref{besselforlac8subfoest}, and thus of \eqref{besselforlac8}.

We next take care of the overlapping size. The argument is essentially the same as before, with two differences: a simplification arises due to the fact that the contribution to each tree arises only from one tile-like boxes, namely the top of the tree $R_\T\times \omega_{\T}$; there is however also a technical complication to our argument arising from the fact that the boxes $R_\T\times \omega_{\T}$ are not tiles, in general. In particular, $\omega_\T$ is not in general an element of the grid $\D_0$. We explain how to overcome this technicality below.

Recall that $\P^{*}$ is what is left of the initial $\P_0$, after the algorithm described above was performed (for both upper and lower lacunary trees). Note that $\P^{*}$ is convex.

 For each overlapping tree $(\T,\xi_\T,R_\T)$ in $\P^{*}$, each $v\in 10\omega_\T$ and each $s\ge 0$ define
$$\omega_{\T,s,v}^{+}:=\{v+2^{-s}10|\omega_\T|,v+2^{2-s}10|\omega_\T| \}.$$ 
$$\omega_{\T,s,v}^{-}:=\{v-2^{2-s}10|\omega_\T|,v-2^{-s}10|\omega_\T| \}.$$
Note that $10\omega_\T\setminus\{v\}\subset \bigcup_{s\ge 0}(\omega_{\T,s,v}^{+}\cup\omega_{\T,s,v}^{-}),$ and that each $m_\T$ adapted to $10\omega_\T$ which vanishes at $v\in 10\omega_\T$ can be written as 
$$m_\T=\sum_{s\ge 0}2^{-s}m_{\T,s,v}^{+}+\sum_{s\ge 0}2^{-s}m_{\T,s,v}^{-},$$
where $m_{\T,s,v}^{+}$ is adapted to $\omega_{\T,s,v}^{+}$ and supported in $10\omega_\T$ while $m_{\T,s,v}^{-}$ is adapted to $\omega_{\T,s,v}^{-}$ and supported in $10\omega_\T$.

It suffices to guarantee that after the elimination process ends we are left with no overlapping trees $(\T,\xi_\T,R_\T)$, no $s\ge 0$, no $n\ge 0$, no $v\in 10\omega_\T$ and no $m_{\T,s,v}$ which is  either adapted to $\omega_{\T,s,v}^{-}$ and supported in $10\omega_\T$ or adapted to $\omega_{\T,s,v}^{+}$ and supported in $10\omega_\T$ such that

\be{selbadtreeoverlapping}
\frac{1}{|R_\T|^{1/2}}\|\tilde{\chi}_{R_\T,n}^{10}(x,y)T_{m_{\T,s,v}}(F_i(\cdot,y))(x)\|_{L^2_{x,y}}\ge 2^{-n-10}\lambda.
\end{equation}

The selection process goes as follows: we will run the following algorithm for each $s\ge 0$. We first run it for $s=0$, and then we increment the value of $s$ and run the algorithm again.

\begin{itemize}
\item Step 0: Set $n=0$, $\P^{**}:=\P^{*}$ and $\F_{bad,n,s,-}=\emptyset$.
\item Step 1: Select a ``bad'' tree  $(\T,\xi_\T,R_\T)$ in $\P^{**}$, that is an overlapping  tree satisfying \eqref{selbadtreeoverlapping} for some $v=v_\T\in 10\omega_\T$ and some $m_{\T,s,v}$ which is adapted to $\omega_{\T,s,v}^{-}$ and supported in $10\omega_\T$. Moreover, we select the tree with minimal $v_\T$.
 Put this tree in the collection $\F_{bad,n,s,-}$. If no such tree is available, go to Step 4.
\item Step 2: Construct the collection $\U((\T,\xi_\T,R_\T),n,s)$ to consist of the following convex trees: for each $J$- dyadic square $R\subset 2^{n+3}R_\T$ with sidelength equal to that of $R_\T$ (and there are $2^{2n+6}$ of them) we let $\T_R$ be the maximal tree with top data $(\xi_\T,R)$; the collection $\U((\T,\xi_\T,R_\T),n)$ will consist of all these (at most $2^{2n+6}$ trees). Eliminate all these convex trees from $\P^{**}$, that is, reset $\P^{**}:=\P^{**}\setminus \bigcup_{\T'\in \U((\T,\xi_\T,R_\T),n)}\T'$
\item Step 3: Go to Step 1
\item Step 4: Reset $n:=n+1$, $\F_{bad,n,s,-}=\emptyset$,  and go to Step 1. 
\end{itemize}
As before, while the algorithm runs forever for each given $s$, it will produce no bad trees for large enough $n$, since $\P^{**}$ is finite. After we are done with eliminating the trees for a given $s$, if $\P^{*}$ is nonempty, we repeat the algorithm with $+$ replacing $-$ and the same $s$ (the only difference is that in Step 1, $v_\T$ will be maximal). We then increment $s$ and repeat the above procedure.

We end up with the collections $(\F_{bad,n,s,-})_{n,s\ge 0}$ and $(\F_{bad,n,s,+})_{n,s\ge 0}$ of trees.

 The Bessel inequality for the selected trees will follow by an argument very similar to the one above for lacunary trees, and from the following observation:

If $(\T,\xi_\T,R_\T)$ and $(\T',\xi_{\T'},R_{\T'})$ are trees in $\F_{bad,n,s,-}$ for given $s,n$, then the boxes $\omega_{\T,s,v_\T}^{-}\times 2^{n+1}R_\T$ and  $\omega_{\T',s,v_{\T'}}^{-}\times 2^{n+1}R_{\T'}$ are disjoint.

The proof of the above goes by contradiction. Assume the two boxes intersect. Without loss of generality we may assume $|R_{\T'}|\le |R_\T|$. Let $P\in\T$, $P'\in\T'$ such that\footnote{Such $P$ and $P'$ must exist since trees are by definition non-empty} $\xi_\T\in 2\omega_P$ and $\xi_{\T'}\in 2\omega_{P'}$. We distinguish two cases:

The first possibility is that $|R_{\T'}|=|R_\T|$. Since $|\xi_\T-\xi_{\T'}|\le 20 \min\{|\omega_P|,|\omega_{P'}|\}$ and due to Remark \ref{moreinfoforoverlap}, it follows that $\bar{\omega}_\T\subset\bar{\omega}_{P'}$ and $\bar{\omega}_{\T'}\subset\bar{\omega}_{P}$. On the other hand, we see that since $2^{n+1}R_{\T}\cap 2^{n+1}R_{\T'}\not=\emptyset$ we have that $R_{P'}\subseteq R_{\T'}\subset 2^{n+3}R_\T$ and $R_{P}\subseteq R_{\T}\subset 2^{n+3}R_{\T'}$. These two facts imply that if $\T$  was selected first, then $P'$ would have qualified to be eliminated at that stage, and hence it would have not been available when $\T'$ was selected. The symmetric statement also holds, and the contradiction arises.

The second possibility is that $|R_{\T'}|<|R_\T|$. Since $\omega_{\T,s,v_\T}^{-}\cap \omega_{\T',s,v_\T'}^{-}\not=\emptyset$ it easily follows that $\T$ was selected first. By reasoning as above and by using the fact that $|\omega_\T|\le 2^{-J}|\omega_{P'}|$ it follows that $\bar{\omega}_\T\subset\bar{\omega}_{P'}$. On the other hand, we see that since $2^nR_{\T}\cap 2^nR_{\T'}\not=\emptyset$ we have that $R_{P'}\subseteq R_{\T'}\subset 2^{n+2}R_\T$, which means $P'$ should have been eliminated at the same stage $\T$ was eliminated. The contradiction arises again. 

This ends the proof of  Proposition \ref{selection} in the case $i\in\{1,2\}$. The case $i=3$ is entirely classical. We are now dealing with two dimensional trees and two dimensional sizes that generalize naturally their corresponding one dimensional counterparts. In short, we successively eliminate maximal trees (with no distinction between lacunary and overlapping this time, no minimality assumptions on $\xi_\T$). The fact that these trees will form a two dimensional forest (in particular they are 3-lacunary) follows from Remark \ref{3lacanyways} and from the fact that each selected tree is maximal. We omit the details.

\subsubsection{Phase-space projections}
\label{sec:7}
Throughout this section we will assume $(\T,\xi_\T,R_\T)$ is a convex tree with $|R_T|=1$, so that $j_\T=0$, and $\xi_\T=0$. We also work with $F_i$ be as in \eqref{supbound}. We will consider $p_i>2$ and denote by $\theta_i=\frac{2}{p_i}$ for $i\in\{1,2\}$, while  $\theta_3$ will be an arbitrary number in the interval $[0,1]$, independent of $p_1,p_2,p_3$. Our goal is to construct phase-space projections associated with each function $F_i$ and the tree $\T$, and to prove  Proposition \ref{phase-space}. In doing so, we follow the terminology and approach from \cite{MTT8}.

We note that
 $\tilde \chi_{j}$ has Fourier support in the region $\{(\xi,\eta):  |\xi| \le 2^{jJ-J} \}$ while $\tilde{\tilde{\chi_{j}}}$ in the region $\{(\xi,\eta):  |\xi|,|\eta| \le 2^{Jj-J} \}$.  

Let $w$ be a positive function on $\R$ and $r > 0$ be a number.  We say that $w$ is \emph{essentially constant at scale $r$} if there is a constant $C=O(1)$ such that
\be{essconst}
C^{-1}(1 + \frac{|x-y|}{r})^{-100} \le \frac{w(x)}{w(y)}
\le C (1 + \frac{|x-y|}{r})^{100}
\end{equation}
for all $x, y \in \R$.  In particular, the weights $\tilde \chi_I^\alpha$ are essentially constant at  scale $|I|$ or less when $|\alpha| \leq 100$.

We shall need the following weighted version of Bernstein's inequality, see \cite{MTT8}.

\begin{lemma}\label{bernstein}
Let $f:\R\to\C$ be a function whose Fourier transform is supported on an interval $\omega$ of width $O(2^{jJ})$ for some integer $j$. Then we have
$$ \| w f \|_\infty \lesssim 2^{ jJ/2} \| w f \|_2$$
for all weights $w$ which are essentially constant at scale $2^{-jJ}$. The implicit constant in the above inequality only depends on the constant $C$ from \eqref{essconst}.
\end{lemma}

\begin{proof}
We can write $f = T_m f$ where $m$ is a suitable bump function adapted to $2\omega$.  From the decay of the kernel of $T_m$ we thus have the pointwise estimate
$$ |f(x)| = |T_m f(x)| \lesssim 2^{jJ} \int \frac{|f(y)|}{(1 + 2^{jJ} |x-y|)^{10}}\ dy$$
and the claim easily follows.
\end{proof}

\begin{lemma}
\label{count}
Let $(\T,\xi_\T,R_\T:=I_\T\times J_\T)$ be a selected tree and $j,j'\in\J_\T$ with $j<j'$. Then
$$E_{j',\T}\subseteq E_{j,\T}.$$
Also,
$$\|\sum_{j\in \J_\T}2^{-jJ}\#\partial E_{j,\T,x}\|_{L^{\infty}_x(I_\T)}\lesssim |I_\T|$$
(with a similar statement for $y$), where
$$E_{j,\T,x}:=E_{j,\T}\cap (\{x\}\times\R).$$
\end{lemma}
\begin{proof}
The argument is the same as in Lemma 4.8 in \cite{MTT8}, by noting that the cross sections $E_{j,\T,x}$ share the same properties as $E_{j,\T}$.
\end{proof}
Recall that Lemma \ref{count} was used in Section \ref{sec:22} to replace the spacial truncations in Proposition \ref{mainsingletreeestimate5}  by certain smoother variants of themselves, thus reducing the proof of that proposition to proving \eqref{supbound177654}.


Let now $(\T,\xi_\T,R_\T)$ be a (not necessarily convex) tree. We also need to work with the following variants $\tilde{E}_j$ of $E_{j,\T}$ which enjoy better regularity properties.

\begin{definition}
\label{hull-definition}
Let $\R_\T$ be the collection of all
maximal dyadic squares $R\subseteq R_T$ which have the property that
$3 R$ does not contain any of the squares $R_P$ with $P\in \T$.
For an integer $j\ge 0$ let $\tilde{E}_j$ be the union of all squares $R$
in $\R_\T$ such that $j_R>j$. 
For an integer $j$ with $j<j_T$ we define $\tilde{E}_j=\emptyset$.
\end{definition}

The sets $\tilde{E}_j$ obviously depend on the tree $\T$, but we suppress
this dependence.

Clearly the intervals in $\R_\T$ form a partition of $R_\T$ and the sets  $\tilde{E}_j$ are nested.
The nice regularity properties are stated in the following lemma:

\begin{lemma}\label{nicer-lemma}
Any two neighboring squares in $\R_\T$ differ by at most a factor
$2^{J}$ in their sidelength ratio. 

The set $\tilde{E}_j$ is a union of $J$- dyadic squares of sidelength
$2^{-jJ}$ and contains $E_{j,\T}$ if $j\in \J_\T$.
\end{lemma}

The following two lemmas will be mainly applied together and are the main ingredient behind the estimates on phase-space projections in Proposition \ref{phase-space}.

\begin{lemma}
\label{auximp21}
If  $R_0$ is a $J$- dyadic square with sidelength $2^{-j_0J}$ such that $3R_0\cap \tilde{E}_{j_0}\not=\emptyset$, then there is $P\in\T$  such that $R_{P}\subseteq 10R_0$ and $j_P\ge j_{R_0}$.
\end{lemma}
\begin{proof}
There is a dyadic square $R_1$ with sidelength $2^{-j_0J}$ which is contained in $\tilde{E}_{j_0}\cap 5R_0$. By the definition of $\tilde{E}_{j_0}$, $3R_1$ contains some $R_P$, with $P\in\T$. Since $3R_1\subset 7R_0$, it follows that $R_P\subset 7R_0$. The claim now follows from the fact that both $R_0$ and $R_P$ are $J$- dyadic.
\end{proof}

\begin{lemma}
\label{auximp22}
We have
\be{unsigned-estlacunary}
\| \tilde \chi_{R_P}^{10}(x,y) T_{m_{P_i}}(F_i(\cdot,y))(x,y)\|_2
 \lesssim |R_P|^{1/2} \size^*_i(\T) 
\end{equation}
for all indices $1 \leq i \leq 2$, all trees $\T$, multi-tiles $P \in \T$ and symbols $m_{P_i}$ adapted to $\omega_{P_i}$ of order $N^2$. A similar statement holds for $i=3$, with the obvious modifications.

We also have
\be{unsigned-est}
\| \tilde \chi_{R_P}^{10}(x,y) T_{m_{P_i}}(F_i(\cdot,y))(x,y)\|_2
 \lesssim |R_P|^{1/2} \size^*_i(\T) 
\end{equation}
for all indices $1 \leq i \leq 2$, all trees $\T$, multi-tiles $P \in \T$ and symbols $m_{P_i}$ adapted to $10 \omega_{P_i}$ of order $N^2$ which in addition vanish at some point $v_P\in 10\omega_P$. 

Moreover,
\be{unsigned-est-sup}
\| \tilde \chi_{R}^{10}(x,y)T_{m_{R}}(F_i(\cdot,y))(x,y) \|_2
 \lesssim |R|^{1/2} \size^*_i(\T) 
\end{equation}
for all indices $1 \leq i \leq 2$, all non-lacunary trees $\T$, 
all $J$-dyadic intervals $R$ for which there is $P\in\T$ with $R_P\subseteq 10 R$ and all
symbols $m_{R}$ adapted to $5\omega_{\xi_\T,R}$. 
\end{lemma}
\begin{proof}
Inequality \eqref{unsigned-estlacunary} follows from the fact that $(\{P\},\xi,R_P)$ is a lacunary tree for some appropriate $\xi$, see Remark \ref{bothlacandoversingle}.

Similarly, \eqref{unsigned-est} follows from the fact that $(\{P\},c(\omega_P),R_P)$ is an overlapping tree and the fact that $\omega_{c(\omega_P),R_P}=\omega_P$.

Now we consider \eqref{unsigned-est-sup}.
Observe first of all that $|R_P|\le |R|$
because both squares are $J$-dyadic. By translating $R$, we may
as well assume $R_P\subseteq R$. Namely, we have to
translate $R$ by at most ten times its length, and observe that
$\tilde{\chi}_R$ stays the same up to some bounded factor.

We consider the two cases $|R_P|=|R|$ and
$|R_P|<|R|$. 

Assume first $|R_P|=|R|$ and thus $R_P=R$.
Observe that by non-lacunarity and the fact that $|\omega_{P}|=|\omega_{\xi_T,R}|$, it follows that $\omega_{P}$
is strictly contained in $5\omega_{\xi_T,R}$.
Hence $5\omega_{\xi_\T,R}$ is contained in
$10\omega_{P}$ and the two intervals are comparable in size.
The claim now follows from \eqref{unsigned-est} since $m_R$ also vanishes at any $v\in 10\omega_{P}\setminus \omega_{\xi_\T,R}$.

Now assume $|R_P|<|R|$. We consider again the singleton
tree $\{P\}$ with top data $(\xi,R)$ so that 
$\xi$ is an endpoint of $5\omega_{\xi_\T,R}$.
Again, by non-lacunarity we see that
these top data indeed turn $\{P\}$ into a tree.
Note that $5\omega_{\xi_\T,R}\subsetneq 10\omega_{\xi,R}$ and the two intervals are comparable in size. It follows that the multiplier $m_{R}$ is adapted to $10\omega_{\xi,R}$
(with a possibly larger constant). Moreover, $m_{R}$ vanishes at any $v\in10\omega_{\xi,R}\setminus5\omega_{\xi_\T,R}$ and \eqref{unsigned-est-sup} follows 
by definition of the tree size.

\end{proof}
\begin{remark}
Inequality \eqref{unsigned-estlacunary} controls projections associated with a multi-tile, and intervenes in the estimates for the phase-space projection in the lacunary case (i.e. the case of lacunary tree when $i\in\{1,2\}$, and the case of a general tree when $i=3$). See the proof Proposition \ref{phase-space}. It was also used (via Lemma \ref{triv-bounds} below) to replace functions by their phase-space projections in the model sum associated with a tree. See Section \ref{sec:22}.

On the other hand, \eqref{unsigned-est-sup} controls the phase-space projection on a one dimensional tile-like region which is not necessarily a one dimensional tile. This estimate will be used repeatedly in the estimates for an overlapping tree in Proposition \eqref{phase-space}.
\end{remark}

An easy application of the above lemma gives
\begin{lemma}\label{triv-bounds}
For all $1 \leq i \leq 3$, $j \in \J_\T$, and $J$- dyadic squares $R$ with $j_R=j$ we have
$$ \|\tilde{\tilde \chi}_j^{1/6} \tilde \pi_j F_i \|_{L^{p_i}(R)} \lesssim
|R|^{1/p_i} \size^*_i(\T)^{\theta_i}.$$
\end{lemma}
\begin{proof}
By interpolation  it suffices to prove the bounds
$$ \|\tilde{\tilde \chi}_j^{1/6} \tilde \pi_j F_i \|_{L^2(R)} \lesssim
|R|^{1/2} \size^*_i(\T),$$
$$ \|\tilde{\tilde \chi}_j^{1/6} \tilde \pi_j F_i \|_{L^{\infty}(R)} \lesssim 1,$$
and (in the $i=3$ case only)
$$ \|\tilde{\tilde \chi}_j^{1/6} \tilde\pi_j F_3\|_{L^{\infty}(R)} \lesssim \size^*_3(\T).$$
The second estimate is immediate from the boundedness of the $F_i$, while the third follows from the first and (the 2 dimensional version of) Lemma \ref{bernstein}.  

Thus it suffices to prove the first inequality.
Fix $i$, $j$, $R$.  There exists $P \in \T$ with $|R_P|=|R|$ 
such that we have the pointwise estimate
$$\tilde{\tilde \chi}_j(x,y)^{1/6} \lesssim \tilde {\chi}_{R_P}^{10}(x,y) $$ 
on $R$.  It thus suffices to show that
$$  \|\tilde{\chi}_{R_P}^{10} \tilde \pi_j F_i \|_{L^2} \lesssim |R|^{1/2} \size^*_i(\T).$$
This however was observed in \eqref{unsigned-estlacunary}.

\end{proof}

For each $y\in J_T$ we define the cross sections
$$\tilde{E}_{j,y}:=\tilde{E}_{j}\cap (\R\times\{y\}).$$

We also denote with $\tilde{\Omega}_j$  the collection of connected components of $\tilde E_j$. Note that such a  $U\in\tilde{\Omega}_j$ may not necessarily be a square, however, from Lemma \ref{nicer-lemma} we know it is a union of dyadic squares of sidelength $2^{-jJ}$. It follows that each such $U$ can be decomposed as a disjoint union of dyadic rectangles $I\times J$, such that $|J|=2^{-jJ}$,  $I$ is a $J$-dyadic interval whose length is an integer multiple of $2^{-jJ}$ and such that the line segments $x_I^l\times J$ and $x_I^r\times J$ lie on the boundary of $\tilde E_j$ (or equivalently, on the boundary of $U$). Here  $x_I^l$ and $x_I^r$ are the left and right endpoints of $I$, respectively. We denote with $\Omega_j$ the collection of all such rectangles $I\times J$ that arise by decomposing each $U\in\tilde{\Omega}_j$.

\begin{lemma}\label{count-again}
Let $(\T,\xi_\T,R_\T)$ be a (not necessarily convex) tree. Then for each $y\in J_\T$ which is not a dyadic point\footnote{The set of dyadic points, i.e. endpoints of dyadic intervals, form a set of measure zero, so restricting to their complement will not affect the later part of our argument. Sets of measure zero will be repeatedly ignored in the following.}
\begin{equation}
\label{reguboundary}
 \sum_{j \in \J_\T} 2^{-jJ} \# \partial \tilde {E}_{j,y} \lesssim |I_\T|\ \  ,
\end{equation}
with the implicit constant independent of $y$.

For each $R=I\times J \in \Omega_j$,  and let $I^l_j$ and $I^r_j$ denote the intervals
$$ I^l_j := (x^l_I - 2^{-jJ-1}, x^l_I - 2^{-jJ - 2})$$
$$ I^r_j := (x^r_I + 2^{-jJ - 2}, x^r_I + 2^{-jJ-1})$$
and define $R^l_j:=I^l_j\times J$ and $R^r_j:=I^r_j\times J$

Then the intervals $R^l_j$ are disjoint as $j$ varies in the integers
with $j\ge 0$ and $R$ varies in $\Omega_j$, 

Moreover, for any two such intervals $R_j^l$ and ${R'}_{j'}^l$ their "horizontal" distance\footnote{This distance is the same if measured at any $y\in J\cap J'$} 
$$\min_{(x,y)\in R_j^l,\;(x',y)\in {R'}_{j'}^l}|x-x'|$$
is at least $\max\{2^{-(jJ+3)},2^{-(j'J+3)}\}$. Similar statements hold for the rectangles $R_j^r$.
\end{lemma}

\begin{proof}
The proof of the lemma is a reprise of the arguments involved in the proof of Lemma 4.12 in \cite{MTT8}. The underlying philosophy is that for each $y$ the sets $\tilde {E}_{j,y}$ inherit much of the properties of the sets $\tilde {E}_{j}$. To illustrate this principle, we will sketch the argument. 
 
Fix $y$ which is not a dyadic point and let $\I_{y,j}$ be the collection of the connected components of $\tilde {E}_{j}\cap (\R\times\{y\})$. This collection is nothing else than the collection of intervals 
$$\{R\cap(\R\times\{y\}):R\in\Omega_j\}.$$
All the claims of the lemma will follow if we prove that $I\in \I_{y,j}$, $I'\in \I_{y,j'}$ and $j'\ge j$ imply that the distance between $I^l_j$ and ${I'}^l_{j'}$ is at least $2^{-jJ-3}$. Indeed, this will imply in particular the disjointness of $I^l_j$ and ${I'}^l_{j'}$, which in turn, will imply \eqref{reguboundary} (since all $I^l_j$ are contained in $3I_\T$). 

Let now $I$ and $I'$ as above, it remains to prove the claim about the distance. 
 Let $J$ be the dyadic interval of length $2^{-jJ}$ containing $y$. Thus $I\times J$ is an element of $\Omega_j$. Let $R_0=I_0\times J$ be the unique dyadic square of sidelegth $2^{-jJ}$ such that the right endpoint of $I_0$ coincides with the left endpoint of $I$. By the definition of $\Omega_j$ it will follow that $R_0\cap \tilde{E}_j=\emptyset$. Due to nestedness, this implies $R_0\cap \tilde{E}_{j'}=\emptyset$. The claim follows.

\end{proof}
As an immediate consequence of the above proposition we have that 
\begin{equation}
\label{reguboundaryforbigrectangles}
 \sum_{j \in \J_\T}\sum_{R\in\Omega_j}(|R_j^l|+|R_j^r|)  \lesssim |R_\T|.
\end{equation}

We introduce for each $j\ge 0$ and $2<p<\infty$ the weight function
$$\mu_{j,p}(x,y)=\sum_{j'\ge 0}2^{-\frac{|j'-j|}{100p}}\sum_{z\in \partial\tilde{E}_{j',y}}(1+2^{j'J}|x-z|)^{-100}.$$
This function will be used to quantify the extra gain we obtain when interacting different scales. 

Note that due to Lemma \ref{count-again} we have $\sup_{j}\|\mu_{j,p}\|_{L^{\infty}}\lesssim_p 1$ and moreover, due to \eqref{reguboundary} we have
\be{weil1}
\int_{R^2}\sum_{j}\mu_{j,p}\lesssim_p |R_T|
\end{equation}

We will construct the associated phase-space projections as follows:

\begin{proposition}
\label{phase-space}
For each $1\le i\le 3$ there exists a function $\Pi_i(F_i)$ such that 
\begin{itemize}
\item (Control by size) 
\begin{equation}
\label{fire-1}
\|\Pi_i(F_i)\|_{p_i}\lesssim |R_\T|^{1/p_i}\size_i^*(\T)^{\theta_i}
\end{equation}
\item ($\Pi_3(F_3)$ approximates $F_3$ on $\T$): For each $j\in \J_\T$ we have
\begin{equation}
\label{lacunary13499}
\tilde{\tilde{\chi}_j}\tilde{\pi}_jF_3=\tilde{S}_j\Pi_3(F_3)
\end{equation}
where $\tilde{S}_j$ is a suitable 2 dimensional Littlewood-Paley projection  to the frequency region $\{(\xi,\eta): |\xi|\le C_02^{(j+1)J},2^{jJ-J}\le|\eta|\le 2^{jJ+J}\}$
\item ($\Pi_i(F_i)$ approximates $F_i$ on $\T$, $i\in\{1,2\}$; the lacunary case): Assume $\T$ is lacunary. Then for each $i\in\{1,2\}$, $j_0\in \J_\T$
\begin{equation}
\label{lacunary13499ttgdt}
{\tilde{\chi}_j}\tilde{\pi}_jF_i=\tilde{S}_j\Pi_i(F_i)
\end{equation}
where $\tilde{S}_j$ is a suitable 1 dimensional Littlewood-Paley projection  to the frequency region $\{\xi: 2^{(j-1)J}\le |\xi|\le 4000\times2^{(j+1)J}\}$
\item ($\Pi_i(F_i)$ approximates $F_i$ on $\T$, $i\in\{1,2\}$; the overlapping case): Assume $\T$ is overlapping. Then for each $i\in\{1,2\}$, $j_0\in \J_\T$ and all $J-$ dyadic squares $R_0=I_0\times J_0$ with $j_{R_0}=j_0$ we have
\begin{equation}
\label{error-1}
\|\tilde{\chi}_{j_0}^{1/6}\tilde{\pi}_{j_0}(F_i-\Pi_i(F_i))\|_{L^{p_i}(R_0)}\lesssim \size_i^{*}(\T)^{\theta_i}|R_0|^{1/p_i-1}\int_{R_0}\tilde{\chi}_{R_0}^2\mu_{j_0,p_i}
\end{equation}
\item (Local control by size) For each $1\le i\le 3$, $j_0\in \J_\T$ and each $J$- dyadic $R_0$ with $j_R=j_0$ we have
\begin{equation}
\label{lacunary13497}
\|\tilde{\tilde \chi}_{j}^{1/6}\tilde{\pi}_{j}(\Pi_i(F_i))\|_{L^{p_i}(R_0)}\lesssim \size_i^{*}(\T)^{\theta_i}|R_0|^{1/p_i}.
\end{equation}
\end{itemize}
\end{proposition}
\begin{proof}
We define 
$$\Pi_3(F_3):=\sum_{j\in\J_\T}\tilde{\tilde\chi_j}\tilde\pi_jF_3,$$
and in the case the tree is lacunary and $i\in\{1,2\}$
$$\Pi_i(F_i):=\sum_{j\in\J_\T}{\tilde\chi_j}\tilde\pi_jF_i.$$
As mentioned before, we only need to smoothen out the spatial component which is associated with a frequency projection, and we do that in order to preserve frequency localization, and thus orthogonality. If no frequency projections are present (and this is the case with the $y$ component of both $F_1$ and $F_2$), then rough spatial cutoffs suffice.

The proof that the above projections satisfy all the required estimates follows exactly the same lines as the proof in the lacunary case of Proposition 7.4 in \cite{MTT8}. This is a fairly easy exercise compared to the overlapping case presented next, since (for $J$ large enough) the portions of the projections corresponding to different scales are pairwise orthogonal. We omit the details.

We now construct the projections in the case $i\in\{1,2\}$, when the tree is assumed overlapping. The construction follows again closely the lines of that in the non-lacunary case in Proposition 7.4 in \cite{MTT8}, with a few modifications. We include the argument for completeness.

We start by defining the projection. The projections turn out to be identical for $i=1$ and $i=2$. We will drop the $i$ index on the function $F_i$.

For each real number $j$, let $T_j$ be a one dimensional Fourier multiplier (defined, say, by
dilations of a fixed multiplier) whose symbol is supported in the frequency 
region 
$\{|\xi| \le 2^{2+jJ}\}$, and equals 1 for $\{|\xi| \le 2^{1+ jJ} \}$.     
Let $S_j$ be the associated Littlewood-Paley  projections $S_j := T_j - T_{j-1}$. We may assume that the kernels of $T_j$ and $S_j$ are real and even. These multipliers will act on the first variable of  functions on $\R^2$.

A first guess as to the construction of $\Pi_i(F)$ would be
$$ \tilde \Pi_i(F)(x,y):= \chi_{\tilde E_{0}} T_{j(x,y)}F,$$
where for each $(x,y) \in \tilde E_{0}$ we define the integer-valued function $j(x,y)$ by
$$j(x,y) := \max \{j\ge 0 : (x,y) \in \tilde E_{j} \}$$

One can expand $\tilde \Pi_i(F)$ as a telescoping series:
\begin{align}
\label{telescope}
\tilde \Pi_i(F) &= \chi_{\tilde E_0} T_0 F + \sum_{1 \leq j} \chi_{\tilde E_j} S_{j} F\\&=
\label{telescope1}\chi_{\tilde E_0} T_0 F + \sum_{1 \leq j}\sum_{R\in\Omega_j}\chi_{R} S_{j} F.
\end{align}
This proposed projection turns out to obey \eqref{fire-1}, but does not obey \eqref{error-1} due to the poor frequency localization properties of the characteristic functions $\chi_{R}$ in \eqref{telescope1}.  Specifically, the cutoffs destroy the vanishing moments of the $S_{j}F$, and this will cause a difficulty when trying to sum in $j$ because the projection $\tilde \pi_{j_0}$ is non-lacunary.  

To get around this problem we shall modify each term  $\chi_{R} S_{j} F$ to have a zero mean on each $y$ fiber.  In order that these modifications do not collide with each other, we shall place them in disjoint rectangles, namely in the rectangles $R_j^r$ and $R_j^l$ constructed in Lemma \eqref{count-again}. 

Let $R=I\times J\in \Omega_j$ for some $j\ge 1$.
Let $\phi^l_{I,j}$ and $\phi^r_{I,j}$ be some functions supported in $I^l_j$ and  $I^r_j$, respectively, uniformly bounded by 10 and with total mass 
$$ \int \phi^l_{I,j}(x)dx=\int \phi^r_{I,j}(x)dx=2^{-jJ}.$$

Decompose $\chi_I$ as $\chi_I = H^l_I + H^r_I$, where $H^l_I(x) := H(x - x^l_I)$ and $H^r_I(x) := -H(x - x^r_I)$ are shifted Heaviside functions. For each $y\in J$ define the quantities $c^l_{I,j}(y)$ and $c^r_{I,j}(y)$ by
$$
c^l_{I,j}(y):= 2^{jJ} \int H^l_I(x) S_{j} F(x,y)dx\ \ ,
$$
$$ c^r_{I,j}(y):= 2^{jJ} \int H^r_I(x) S_{j} F(x,y)dx.$$

Define now the functions $\phi^l_{R,j}(x,y):=\phi^l_{I,j}(x)\times c^l_{I,j}(y)$ and $\phi^r_{R,j}(x,y):=\phi^r_{I,j}(x)\times c^r_{I,j}(y)$. Note that they are supported on  $R_j^l$ and $R_j^r$, respectively and that
\be{meanzerofiberwise}
\int[\chi_{R}(x,y) S_{j} F(x,y)-\phi^l_{R,j}(x,y)-\phi^r_{R,j}(x,y)]dx=0
\end{equation}
for each $y\in J$.

We can now define the correct form of the projections
$$\Pi_i(F) := \tilde \Pi(F) - \sum_{1 \leq j} \sum_{R \in \Omega_j} (\phi^l_{R,j} + \phi^r_{R,j}).$$

The control on the functions $\phi^l_{R,j}$ (and a similar control holds on $\phi^r_{R,j}$, too) is provided by the following lemma:

\begin{lemma}\label{clij-est}
Let $R=I\times J \in \Omega_j$.  
Then we have the estimate
\be{big-bound}
|c^l_{I,j}(y)| \lesssim  2^{jJ} 
\int \frac{|S_{j} F(x,y)| dx}{(1 + 2^{jJ} |x-x^l_I|)^{100}},
\end{equation}
\be{clb-2}
\|\phi^l_{R,j}\|_2\lesssim \size_i^{*}(\T)|R_j^l|^{1/2}
\end{equation}
and
\be{clb}
\|\phi^l_{R,j}\|_{\infty}\lesssim 1
\end{equation}
\end{lemma}
\begin{proof}
Estimate \eqref{big-bound} was proved in Lemma 8.1 in \cite{MTT8}. Note also that \eqref{clb} is a consequence of \eqref{big-bound} and the fact that $\|F\|_{\infty}\le 1.$ 

It remains to prove \eqref{clb-2}. Denote by $R'$ the $J$-dyadic square  $[x_I^l,x_I^l+2^{-jJ}]\times J$. Note that $R'$ is a subset of $\tilde{E}_j$. By Lemma \ref{auximp21} it follows that there exists $P\in\T$ with $R_P\subseteq 10R'$. Using this and the fact that $S_j$ is associated with a multiplier adapted to $5\omega_{\xi_\T,R'}$, \eqref{clb-2} will now follow from \eqref{unsigned-est-sup}.
\end{proof}

We now prove \eqref{fire-1} in the case $i\in\{1,2\}$ and the tree is overlapping. It suffices by interpolation to prove that
\be{Linftyforproj}
\|\Pi_i(F)\|_{\infty}\lesssim 1
\end{equation}
and 
\be{Ltwoforproj}
\|\Pi_i(F)\|_2\lesssim |R_\T|^{1/2}\size_i^{*}(\T).
\end{equation}

Let us start with \eqref{Linftyforproj}. By \eqref{clb} and disjointness of supports, the contribution of $\phi^l_{R,j}$ and $\phi^r_{R,j}$ is acceptable. It remains to prove that 
$$\|\tilde{\Pi}(F)\|_{\infty}=\|T_{j(x,y)}F(x,y)\|_{\infty}\lesssim 1.$$
Note however that this is an immediate consequence of \eqref{supbound}.

Next, we focus on \eqref{Ltwoforproj}. Again, it suffices to prove
\be{rrtyer5}
\|\tilde{\Pi}_i(F)\|_2\lesssim |R_\T|^{1/2}\size_i^{*}(\T)
\end{equation}
and
\be{rrtyer5kpzx}
(\sum_{j\ge 1}\sum_{R\in \Omega_j}(|R_j^l|+|R_j^r|))^{1/2}\lesssim |R_\T|^{1/2}
\end{equation}
(by taking into account \eqref{clb-2} and the disjointness of the supports of
$\phi^l_{R,j}$ and $\phi^r_{R,j}$). The last inequality was however observed in \eqref{reguboundaryforbigrectangles}.

To prove \eqref{rrtyer5}  we expand
$$ \tilde \Pi_i(F) = \sum_{0 \leq j } \chi_{\tilde E_j \backslash \tilde E_{j+1}} T_{j} F
= \sum_{0 \leq j } \sum_{R \cap \tilde E_j \backslash \tilde E_{j+1}\neq \emptyset: j_R = j\atop{R\;is \; J-dyadic}} \chi_{R\cap \tilde E_j \backslash \tilde{E}_{j+1}} 
T_{j} F.$$

As $j$ and $R$ vary in the above sum, the sets 
$R\cap \tilde E_j\backslash \tilde E_{j+1}$ are pairwise disjoint, hence it suffices to show
$$ 
\sum_{0 \leq j } \sum_{R \cap \tilde E_j \backslash \tilde E_{j+1}\neq \emptyset: j_R = j\atop{R\;is\; J-dyadic}} 
\| T_{j} F \|_{L^2(R)}^2 \lesssim |R_\T| \size^*_i(\T)^2$$

For each $j$ and $R$ in this sum there is a $J$- dyadic 
square $R'\subseteq R$ with $j_{R'}=j_R+1$ which is contained in
$\tilde E_j \backslash \tilde E_{j+1}$. 
This follows from Lemma \ref{nicer-lemma}.
As $j$ and $R$ vary, these intervals $R'$ are pairwise disjoint. Hence
it suffices to show that
$$ \| T_{j} F \|_{L^2(R)} \lesssim |R|^{1/2} \size^*_i(\T)$$
for all $j$, $R$ in the above sum.  But for such $j$, $R$ we can find a multi-tile 
$P \in \T$ with $R_P\subseteq 10R$ by Lemma \ref{auximp21}.  
The claim then follows from \eqref{unsigned-est-sup}, since $T_j$ is associated with a multiplier adapted to $5\omega_{\xi_\T,R}$.  
This proves \eqref{rrtyer5} and thus \eqref{Ltwoforproj}.
The proof of \eqref{fire-1} is now complete.

The estimate \eqref{lacunary13497} will follow from \eqref{error-1}, Lemma \ref{triv-bounds}, the triangle inequality, and the fact that the $\mu_{j,p}$ are uniformly bounded.  Thus it only remains to verify \eqref{error-1}.

Fix $j_0\ge 0$ and $R_0$ such that $j_{R_0}=j_0$.  From the frequency support of $\tilde \pi_{j_0}$ we may replace $F - \Pi_i(F)$ with $T_{j_0} F - \Pi_i(F)$.

We now decompose
\begin{align}
T_{j_0} F - \Pi_i(F) 
=& \chi_{\R \backslash \tilde E_{j_0}} T_{j_0} F 
\label{d1}\\
&- \chi_{\R \backslash \tilde E_{j_0}} \tilde \Pi_i(F)
\label{d2}\\
&+ \chi_{\R \backslash \tilde E_{j_0}} 
\sum_{1 \leq j \leq j_0} \sum_{R \in \Omega_j}\phi^l_{R,j}
\label{d3}\\
&+ \chi_{\R \backslash \tilde E_{j_0}} \sum_{1 \leq j \leq j_0} \sum_{R \in \Omega_j} \phi^r_{R,j}
\label{d4}\\
&- \sum_{j_0 < j } \sum_{R=I\times J \in \Omega_j} 
G^l_{R,j}\label{d5}\\
&- \sum_{j_0 < j } \sum_{R=I\times J \in \Omega_j} 
G^r_{R,j}.\label{d6}
\end{align}
where 
$$G^l_{R,j}(x,y):=H^l_{I}(x)1_{J}(y)S_{j}F(x,y) - \phi^l_{R,j}(x,y)$$
$$G^r_{R,j}(x,y):=H^r_{I}(x)1_{J}(y)S_{j}F(x,y) - \phi^r_{R,j}(x,y)$$

The first four terms are treated exactly like in \cite{MTT8}. Since they are supported outside $\tilde E_{j_0}$, only the $\mu_{j_0,p_i}^{j\le j_0}$ part of $\mu_{j_0}$ will enter the estimates, where we define
$$\mu_{j_0,p_i}^{j\le j_0}(x,y)=\sum_{0\le j\le  j_0}2^{-\frac{|j_0-j|}{100p_i}}\sum_{z\in \partial\tilde{E}_{j,y}}(1+2^{jJ}|x-z|)^{-100}.$$
Note that $\mu_{j_0,p_i}^{j\le j_0}$ is constant on $J$-dyadic squares with sidelength $2^{-j_0J}$, in particular on $R_0$. The estimates for the first four terms above will follow by interpolation from the following two estimates
$$
\| \tilde\chi_{j_0}^{1/6}\tilde \pi_{j_0} (\eqref{d1} - \eqref{d2} + \eqref{d3} + \eqref{d4}) \|_{L^2(R_0)}\lesssim \frac{1}{|R_0|^{1/2}}\size^{*}(\T)\int_{R_0}\tilde{\chi}_{R_0}^2\mu_{j_0,p_i}^{j\le j_0},$$

$$
\| \tilde\chi_{j_0}^{1/6}\tilde \pi_{j_0}(\eqref{d1} - \eqref{d2} + \eqref{d3} + \eqref{d4}) \|_{L^\infty(R_0)}\lesssim \frac{1}{|R_0|}\size^{*}(\T)\int_{R_0}\tilde{\chi}_{R_0}^2\mu_{j_0,p_i}^{j\le j_0}.$$ 
We omit the details. 

We however sketch the proof of the estimate for the terms \eqref{d5} and \eqref{d6}, since this is the more delicate case. Again, all the needed technology is already in \cite{MTT8}, but some extra care is needed, since in this case $R_0$ interacts with scales $j$ greater than $j_0$. Due to this, it is $\mu_{j_0,p_i}^{j> j_0}:=\mu_{j_0,p_i}-\mu_{j_0,p_i}^{j\le j_0}$ that will enter our estimates. Since $\mu_{j_0}^{j> j_0}$ is no longer constant on $R_0$, for each $j>j_0$ we will have to split $R_0$ in rectangles of the form $I_0\times J'$ with $|J'|=2^{-jL}$ and  get estimates on each of these rectangles, which would then add up to the desired global estimate (i.e. on the whole $R_0$).

The factor $\tilde\chi_{j_0}^{1/6}$ will be useless here, the decay will come from other sources.
More precisely, we will prove that for each $j>j_0$ and each $R:=I\times J\in\Omega_j$ we have
\be{est1localscalej}
\|\tilde{\pi}_{j_0}G^l_{R,j}\|_{L^{\infty}(R_0)}\lesssim \frac{1}{|R_0|}\int_{y\in J}\int_{x\in \R}\tilde\chi_{R_0}^2(x,y)(1+2^{jJ}|x-x_I^l|)^{-100}dxdy
\end{equation}
\be{est2localscalej}
\|\tilde{\pi}_{j_0}G^l_{R,j}\|_{L^{2}(R_0)}\lesssim \frac{2^{(j_0-j)J/2}}{|R_0|^{1/2}}\size_i^{*}(\T)\int_{y\in J}\int_{x\in \R}\tilde\chi_{R_0}^2(x,y)(1+2^{jJ}|x-x_I^l|)^{-100}dxdy
\end{equation}
with similar estimates for $G^r_{R,j}$. If we have these estimates, all we have to do is interpolate them for each $j$, and then add the resulting estimates over all $j>j_0$ and all $R\in\Omega_j$. Note that we do not get any extra decay  of the form $2^{(j_0-j)\epsilon}$ for the $L^{\infty}$ estimate in \eqref{est1localscalej}. We do however obtain such a decay for the $L^{2}$ estimate in \eqref{est2localscalej}, and by interpolation, the decay factor will have a $p_i$ dependence. This is the main source of the $p_i$ dependence of the function $\mu_{j_0,p_i}$.

Fix $j>j_0$ and $R=I\times J\in\Omega_j$. We start by proving an estimate for $G^l_{R,j}$.  Observe from Fourier support considerations that for each $y$
$$ \tilde \pi_{j_0} ((T_{j-1} H^l_I) S_{j} F) = 0,$$
in particular, $(T_{j-1} H^l_I) S_{j} F$ has mean zero for each $y$.  It thus suffices to get estimates for
\begin{align*}
F_{j,R}^l(x,y)&:= G_{j,R}^l(x,y)-(T_{j-1} H^l_I(x)\chi_J(y))S_{j} F(x,y)\\&= [(1 - T_{j-1})H^l_I(x)\chi_J(y)] S_{j} F(x,y)-\phi^l_{R,j}.
\end{align*}
\end{proof}

We aim first at showing that for each $y\in J$ and each $x\in\R$
\be{estgungtoondgt}
|F_{j,R}^l(x,y)|\lesssim 2^{jJ/2} 
(1 + 2^{jJ} |x'-x^l_I|)^{-200}\|\frac{|S_{j} F(x',y)|}{(1 + 2^{jJ} |x'-x^l_I|)^{50}}\|_{L^2(x')}.
\end{equation}

This estimate is clear for the $\phi^l_{R,j}$ part of $F_{j,R,i}^l$, due to \eqref{big-bound}. It remains to prove it for $[(1 - T_{j-1})H^l_I(x)\chi_J(y)] S_{j} F(x,y)$. 

Fix $y\in J$.
From repeated integration by parts we have the pointwise estimate
$$ |(1 - T_{j-1})H^l_I(x)| \lesssim (1 + 2^{jJ} |x-x^l_I|)^{-250}$$
so it suffices to show that
$$
|S_{j} F(x,y)| \lesssim 2^{jJ/2}(1 + 2^{jJ} |x-x^l_I|)^{50}\|\frac{S_{j} F(x',y)}{(1 + 2^{jJ} |x'-x^l_I|)^{50}}\|_{L^2(x')}.$$
Note however that this is an immediate consequence of Lemma \ref{bernstein}.

We next use the pointwise bounds for $F_{j,R}^l(x,y)$ in \eqref{estgungtoondgt} and the fact that $F_{j,R}^l(\cdot,y)$ has mean zero to get the following bound for the antiderivative $\nabla^{-1}F_{j,R}^l$

\be{estgungtoondgtxxxgggrr}
|\nabla^{-1}F_{j,R}^l(x,y)|\lesssim 2^{-jJ/2} 
(1 + 2^{jJ} |x-x^l_I|)^{-199}\|\frac{|S_{j} F(x',y)|}{(1 + 2^{jJ} |x'-x^l_I|)^{50}}\|_{L^2(x')}.
\end{equation} 

We continue by noting that 
$$\tilde{\pi}_{j_0}F_{j,R}^l=\nabla\tilde{\pi}_{j_0}(\nabla^{-1}F_{j,R}^l).$$
An easy computation shows that if $x_0\in I_0$ then
\begin{align}
|\nabla\tilde{\pi}_{j_0}(\nabla^{-1}F_{j,R}^l)(x_0,y)|&\lesssim 2^{2j_0J}\int|\tilde\chi_{I_0}^{2}(x')\nabla^{-1}F_{j,R}^l(x',y)dx'|\\&
\label{finalggest6678}
\lesssim 2^{(2j_0-\frac{j}{2})J}\|\frac{|S_{j} F(x',y)|}{(1 + 2^{jJ} |x'-x^l_I|)^{50}}\|_{L^2(x')}\int\chi_{I_0}^{2}(x')(1 + 2^{jJ} |x'-x^l_I|)^{-199}dx' 
\end{align}

The estimate \eqref{est1localscalej} now follows from the above and the fact that 
$$\|\frac{|S_{j} F(x',y)|}{(1 + 2^{jJ} |x'-x^l_I|)^{50}}\|_{L^2(x')}\lesssim 2^{-jJ/2}.$$

To prove  \eqref{est2localscalej} we first denote by $R'$ the $J$- dyadic rectangle defined as $[x_I^l,x_I^l+2^{-jJ}]\times J$. We further observe that  \eqref{finalggest6678} implies that
\begin{align*}
\|\tilde{\pi}_{j_0}G^l_{R,j}\|_{L^{2}(R_0)}&=\|\tilde{\pi}_{j_0}F^l_{R,j}\|_{L^{2}(R_0)}\\&\lesssim 2^{(\frac{3j_0}{2}-\frac{j}{2})J}\|\frac{S_{j} F(x,y)}{(1 + 2^{jJ} |x-x^l_I|)^{50}}\|_{L^2_{x,y}(\R\times J)}\int\chi_{I_0}^{2}(x)(1 + 2^{jJ} |x-x^l_I|)^{-199}dx\\&\lesssim 2^{(\frac{3j_0}{2}-\frac{j}{2})J}\|\tilde\chi_{R'}^{10}S_{j} F\|_2\int\chi_{I_0}^{2}(x)(1 + 2^{jJ} |x-x^l_I|)^{-199}dx\\&\lesssim 2^{(\frac{3j_0}{2}-\frac{j}{2})J}\size_i^{*}(\T)\int_{j\in J}\int\tilde\chi_{R_0}^{2}(x,y)(1 + 2^{jJ} |x-x^l_I|)^{-199}dxdy\\&\le\frac{2^{(j_0-j)J/2}}{|R_0|^{1/2}}\size_i^{*}(\T)\int_{y\in J}\int_{x\in \R}\tilde\chi_{R_0}^2(x,y)(1+2^{jJ}|x-x_I^l|)^{-100}dxdy
\end{align*}
where the penultimate inequality follows from Lemma \ref{auximp21} and \eqref{unsigned-est-sup}. This end the proof of \eqref{est2localscalej} and of Proposition \ref{phase-space}.

\subsection{The cone $\int \Psi=0$}\ \
\label{sub:2}

Our goal now is to prove \eqref{conedecompuuh} in the case $\int \Psi=0$. We will first note that the techniques developed in Section \ref{sub:1}, combined with the type of analysis that solved the one dimensional Bilinear Hilbert Transform (see \cite{LT1},\cite{LT2}, and also \cite{thielelectures} for a more detailed exposition),  can address this case, too. Here is a brief explanation why.

Note that the model sum in \eqref{conedecompuuh} represents a (one dimensional) Bilinear Hilbert Tranform in the $x$ coordinate. This is due to the special cancellation condition $\Psi=0$. One works with one and a half dimensional trees and phase-space projections as in Section \ref{sub:1}. The difference is that there will be no orthogonality coming from the $y$ component of the third function (as was the case before; in particular trees are not automatically 3-lacunary), but rather from the special localization in the $x$ component (the same type of localization as  in the case of the one dimensional Bilinear Hilbert Transform). The same sizes $\size_i^{*}$ will control phase-space projections of $F_i$ when $i\in\{1,2\}$. The property of being lacunary or overlapping will be determined only by the $x$ component. For each $i\in\{1,2,3\}$ we will have trees which are $i$- overlapping, and they will necessarily be $i'$ lacunary for each $i'\not=i$. Some of these features are present in the alternative argument we present below.      

We choose to present this alternative argument, since it provides a slightly different angle, and since it is "cleaner" for exposition purposes. One of its advantages is that it avoids\footnote{We will be able to discretize in such a way that the phase-space projections enter the picture in a natural way} the technicalities behind phase-space projections, that were present in the previous section. 

\subsubsection{Discretization}\ \
\label{discrtsyshsbsgdtd6}

The collection $\P$ of multi-tiles in this context will consist of $P=I_P\times\omega_{P_1}\times\omega_{P_2}\times\omega_{P_3}$ with the following properties
\begin{definition}
\begin{itemize}\ \

\item Each component $\omega_{P_i}$ of some $P\in \P$ determines uniquely the other two (frequency) components $\omega_{P_j}$ of $P$.
\item $\omega_{P_i}$ are elements of a shifted dyadic grid, while $I_P$ is an element of the standard dyadic grid
\item $|\omega_{P_1}|=|\omega_{P_2}|=|\omega_{P_3}|=|I_P|^{-1}$
\item $|\omega_{P_i}|=2^{Jj}$ for some $j\in\Z$, where $J\in \N$ is a fixed large enough  natural number. Such intervals will be referred to a $J$- dyadic.
\item $|\omega_i|=|\omega_i'|$ and $\omega_i\not=\omega_i'$ imply $\dist(\omega_i,\omega_i')\ge 2^{J}|\omega_i|$
\item If for some $\xi\in\R$ we denote $(\xi_1,\xi_2,\xi_3):=(\xi,\xi,-2\xi)$, then $\xi_i\in 2\omega_{P_i}$ for some $i\in\{1,2,3\}$ implies that $\xi_j\in C_0\omega_{P_j}\setminus 2\omega_{P_j}$ for each $j\not=i$, where $C_0$ is some large universal constant. 
\end{itemize}
\end{definition}

If $P$ is a multi-tile, we denote by $P_i:=I_P\times \omega_{P_i}$ the associated tiles.

Let $\varphi$ be a function whose Fourier transform is adapted to $[-1/2,1/2]$. We will denote by $\varphi_{P_i}$ the wave-packet localized in the tile $I_P\times \omega_{P_i}$, that is
$$\varphi_{P_i}(x)=\frac{1}{|I_P|^{1/2}}\varphi\left(\frac{x-c(I_P)}{|I_P|}\right)e^{ic(\omega_{P_i})x}.$$
We will also use the notation 
$$\psi_J(x)=\frac{1}{|J|^{1/2}}\psi\left(\frac{x-c(J)}{|J|}\right).$$
By using standard reductions, in order to get bounds in this case for \eqref{conedecompuuh}, it suffices to prove the boundedness of the model sum
\begin{equation*}
\int_\R\sum_{P}\frac{1}{|I_P|}|\langle F_1(x',y),\varphi_{P_1}(x')\rangle_{x'}\langle F_2(x',y),\varphi_{P_2}(x')\rangle_{x'}|\sup_{\psi_{J_{y,P}}}|\langle F_3(x',y'),\varphi_{P_3}(x')\psi_{J_{y,P}}(y')\rangle_{x',y'}|dy,
\end{equation*}
where $J_{y,P}$ is the unique dyadic interval of length $|I_P|$ containing $y$, and the supremum above is taken over all $\psi$ with Fourier transform adapted to $[-1/2,1/2]$.

We will change the angle a bit and rewrite the above expression in a slightly different way.

\begin{definition}
A hyper-multi-tile $(P,J)$ is a multi-tile with an extra spatial component $J_P$, where $J_P$ is dyadic and $|I_P|=|J|$.
\end{definition} 
A hyper-multi-tile also has an extra frequency component, that is $[-\frac{1}{2}|J|^{-1},\frac{1}{2}|J|^{-1}]$. Since this component is implicit, we will omit it, and always write $(P,J)$. 

A hyper-multi-tile will serve the purpose of localizing in time-frequency 2 dimensional wave-packets like $\varphi_{P_i}\times \psi_J$.

Let $\P_{hyper}$ be an arbitrary finite collection of hyper-multi-tiles. For each $y$ we denote by $\P_y$ the collections of multi-tiles $P$ such that $(P,J_{y,P})\in\P_{hyper}$. 

To simplify notation, for each y and each $P\in\P_y$ define
$$a_{P_i}(y)=|\langle F_i(x',y),\varphi_{P_i}(x')\rangle_{x'}|$$
for $i\in\{1,2\}$ and
$$a_{P_3}(y)=\frac{1}{|I_P|^{1/2}}\sup_{\psi_{J_{y,P}}}|\langle F_3(x',y'),\varphi_{P_3}(x')\psi_{J_{y,P}}(y')\rangle_{x',y'}|.$$
Define also for each $(P,J)\in\P_{hyper}$
$$b_{P,J}:=\sup_{\psi_{J}}|\langle F_3(x',y'),\varphi_{P_3}(x')\psi_{J}(y')\rangle_{x',y'}|,$$
and 
$$T(F_1,F_2,F_3)(x,y):=\sum_{P\in\P_y}\frac{1}{|I_P|^{3/2}}\prod_{i=1}^3a_{P_i}(y)\chi_{I_P}(x)=$$
\be{dschlgueiwoy;qdmc';WOCKWEO}
=\sum_{(P,J)\in \P_{hyper}}\frac{1}{|I_P|^2}\prod_{i=1}^2a_{P_i}(y)b_{P,J}\chi_{I_P}(x)\chi_{J}(y).
\end{equation}

A standard limiting argument shows that it suffices to prove
\begin{theorem}
\label{newintusedthm}
For each $2<p_i<\infty$ with $\frac{1}{p_1}+\frac{1}{p_2}+\frac{1}{p_3}=1$  we have
$$\|T(F_1,F_2,F_3)\|_{L^1_{x,y}}\lesssim \prod_{i=1}^{3}\|F_i\|_{p_i}.$$
Moreover, the implicit constant in the above inequality does not depend on $\P_{hyper}$.
\end{theorem}

\subsubsection{The proof of Theorem \ref{newintusedthm}}

We now fix $\P_{hyper}$, and will not index any quantity ($T$, $\P_y$, etc) by it.

By further invoking interpolation and the dilation invariance of our operator, it suffices to prove that for each $2<p_i<\infty$ with $\frac{1}{p_1}+\frac{1}{p_2}+\frac{1}{p_3}=\frac{1}{p}$ and $\frac{p}{p_i}<\frac{1}{2}$, and each $\|F_i\|_{p_i}=1$ we  have
$$
|\{(x,y):T(F_1,F_2,F_3)(x,y)>1\}|\lesssim 1.
$$

The way we prove this is by constructing an exceptional set $E\subset \R^2$ with $E\lesssim 1$ such that for some appropriate $t>1$
\be{nouaecuatiehgstsh}
\int\int_{E^c}T(F_1,F_2,F_3)^t(x,y)dxdy\lesssim 1.
\end{equation}
The set $E$ will be constructed in a few stages. To prove \eqref{nouaecuatiehgstsh}, we may and will assume that all $(P,J)$ that contribute to $T$ in \eqref{dschlgueiwoy;qdmc';WOCKWEO} satisfy
\be{82397eefwjcy9h}
I_P\times J\nsubseteq E.
\end{equation}

\begin{definition}
Let $\xi_\T\in \R$ and let $R_\T$ be a $J$- dyadic square. A two dimensional i-tree  with top data $(\xi_\T,R_\T)$ is a collection $\T$ of hyper-multi-tiles with the property that $\xi_\T\in 2\omega_{P_i}$ and $I_P\times J\subseteq R_\T$ for each 
$(P,J):=I_P\times J\times\omega_{P_1}\times\omega_{P_2}\times\omega_{P_3}\in\T$.

A one dimensional  i-tree  with top data $(\xi_\T,I_\T)$ is a collection $\T$ of multi-tiles with the property that $\xi_\T\in 2\omega_{P_i}$ and $I_P\subseteq I_\T$ for each 
$P:=I_P\times\omega_{P_1}\times\omega_{P_2}\times\omega_{P_3}\in\T$
\end{definition}

Note that if $(\T,\xi_\T,R_\T)$ is a two dimensional tree, then for each $y$, its restriction to the fiber above $y$
$$\{P:(P,J_y)\in\T\}$$
is a one dimensional tree with top data $(\xi_\T,I_\T)$, that we denote by $\T_y$. We will refer to it as the tree induced by $\T$.

We further comment on our strategy to prove \eqref{nouaecuatiehgstsh}. 
For each $y$ denote
$$T_y(F_1,F_2,F_3)(x):=T(F_1,F_2,F_3)(x,y)=\sum_{P\in\P_y}\frac{1}{|I_P|^2}a_{P_1}(y)a_{P_2}(y)a_{P_3}(y)\chi_{I_P}(x).$$
Our plan is to get estimates for
$\int_{E_y^c}T_y^t(F_1,F_2,F_3)(x)dx$ outside the fibered exceptional set $E_y:=E\cap (\R\times\{y\})$. These estimates will then be integrated over $y$ to get \eqref{nouaecuatiehgstsh}.

To estimate $T_y$ we will split the collection $\P_y$ into trees, and will make sure that we gain some control over both $a_{P_i}(y)$ and also over the counting function of the tree tops. The basic estimate for an $i$-tree $\T_y\subset\P_y$ will be
\be{jdbncklsjhlweufywerifgoq}
\|\sum_{P\in\T_y}\frac{1}{|I_P|^{1/2}}\prod_{i=1}^{3}a_{P_i}(y)\frac{\chi_{I_P}(x)}{|I_P|}\|_{\BMO_x}\le 
\|(a_{P_j})_{P\in\T_y}\|_{\BMO}\|(a_{P_l})_{P\in\T_y}\|_{\BMO}\|(\frac{a_{P_i}}{|I_P|^{1/2}})_{P\in\T_y}\|_{\infty},
\end{equation}  
where $l\not=j\in\{1,2,3\}\setminus\{i\}$ and 
$$\|(a_{P_j})_{P\in\T_y}\|_{\BMO}:=\sup_{I\;dyadic}\left(\frac{1}{|I|}\sum_{P\in\T_y:I_P\subseteq I}a_{P_j}^2\right)^{1/2}\sim\|\sum_{P\in\T_y}(a_{P_j})^2\frac{\chi_{I_P}(x)}{|I_P|}\|_{\BMO_x}^{1/2}.$$
Note that $\|\cdot\|_\BMO$ majorizes  $\|\cdot\|_\infty$, so it will suffice to control the former, for each $i$.

The index $i=3$ plays a special role. To insure control over $\|(a_{P_3})_{P\in\T_y}\|_{\BMO}$, we will need to look at $\|(a_{P_3})_{P\in\T_y}\|_{\BMO}$ as being the restriction (to the $y$ fiber) of a similar two dimensional quantity (the 3-size). In short, to control $\|(a_{P_3})_{P\in\T_y}\|_{\BMO}$, instead of selecting one dimensional trees in $\P_y$, we will instead select two dimensional trees in $\P_{hyper}$, and then restrict them to $\P_y$. We explain below this procedure.

\begin{definition}
The $3$-size $\size_3(\P_{hyper}^*)$ of a finite collection $\P_{hyper}^*\subseteq \P_{hyper}$ of hyper-multi-tiles is defined as 
$$\sup_{\T\in\P_{hyper}^{*}}\left(\frac{1}{|R_\T|}\sum_{(P,J)\in\T}b_{P,J}^2\right)^{1/2},$$
where the supremum above is taken over all 1-trees and 2-trees.
\end{definition}

The following Bessel type inequality is standard (see also similar results in the previous section).

\begin{lemma}
\label{bessel2dimght86svbns}
Assume that for each $(P,J)\in\P_{hyper}$ the square $I_P\times J$ intersects the complement of the set\footnote{$M_p$ denotes the $L^p$ version of the Hardy-Littlewood maximal function} 
$$E_3:=\{(x,y):M_{p_3}F_3(x,y)\gtrsim 1\}.$$ We assume as before that  $\|F_3\|_{p_3}=1$. Then we can split 
$$\P_{hyper}:=\bigcup_{m\ge 0}\P_{hyper}^m$$
where
$$\size_3(\P_{hyper}^m)\lesssim 2^{-m}$$
and $\P_{hyper}^m$ is the union of a family $\F_m$ of pairwise disjoint $i$-trees $(i\in\{1,2,3\})$ satisfying
\be{hsdfkjhdsfdjhfkj}
\|\sum_{\T\in\F_m}\chi_{R_\T}\|_{BMO}\lesssim 2^{2m},
\end{equation}
\be{supportindivtreetwodim}
R_\T\subset \{(x,y):M_{p_3}F_3(x,y)\gtrsim 2^{-m}\},
\end{equation}
\be{supportindivtreetwodimdhdyeyehdn}
\sum_{\T\in \F_m}|R_\T|\lesssim 2^{(2+p_3)m},
\end{equation}
and, if the tree is 1-tree or 2-tree then 
\be{hdsakh38789473rjfhjkhdjvkdbvjskdhfu}
\|\sum_{(P,J)\in\T}b_{P,J}^2\frac{\chi_{I_P\times J}}{|I_P\times J|}\|_{\BMO}\lesssim 2^{-2m},
\end{equation}
while if the tree is 3-tree, then
$$\|({b_{P,J}^2})_{(P,J)\in\T}\|_{\infty}\lesssim 2^{-2m}.$$
\end{lemma}

We will state a few consequences of the above. For a.e. $y$\footnote{More precisely, for $y$ not a dyadic point} we let $\T_y$ be the one dimensional tree induced by $\T$, and denote by $\F_{m,y}^{3}$ the collection of these trees. 

A standard application of John-Nirenberg's inequality, together with \eqref{hsdfkjhdsfdjhfkj} and \eqref{supportindivtreetwodim} implies that there is $E_m^{**}$ such that $|E_m^{**}|\lesssim 2^{-Mm}$ and 
\be{dhssuayweoiyuicdhsjkalncmkn}
\|\sum_{\T\in\F_m\atop{R_\T\nsubseteq E_m^{**} }}\chi_{R_\T}\|_{\infty}\lesssim 2^{(2+\epsilon)m}.
\end{equation}
Here and in the following $\epsilon$ can be thought of as being as small as we want, while $M$ as large as we want. We put $E_3^{**}:=\bigcup_m E_m^{**}$ in the exceptional set $E$. From this, \eqref{82397eefwjcy9h} and \eqref{dhssuayweoiyuicdhsjkalncmkn} we get
\be{dhssuayweoiyuicdhsjkalncmkn1jdjk}
\|\sum_{\T_y\in\F_{m,y}^3}\chi_{I_{\T_y}}\|_{\infty}\lesssim 2^{(2+\epsilon)m}.
\end{equation}

Another  application of John-Nirenberg's inequality combined with \eqref{hdsakh38789473rjfhjkhdjvkdbvjskdhfu} implies that if $\T\in\F_m$ is 1-tree or 2-tree then
$$\|\sum_{(P,J)\in\T\atop{I_P\times J\nsubseteq E_{3,\T}}}b_{P,J}^2\frac{\chi_{I_P\times J}}{|I_P\times J|}\|_{L^\infty}\lesssim 2^{-(2-\epsilon)m},$$
for some $E_{3,\T}\subset R_\T$ with $|E_{3,\T}|\lesssim 2^{-Mm}|R_\T|$.

An immediate consequence is that
\be{saghhgedyoweibnsmaxcb}
\|\sum_{(P,J)\in\T\atop{I_P\times J\nsubseteq E_{3,\T}\atop_{y\in J}}}\frac{b_{P,J}^2}{|I_P|}\chi_{I_P}\|_{L^\infty(\R)}=\|\sum_{P\in\T_y\atop{I_P\times J\nsubseteq E_{3,\T}}}\frac{a_{P_3}^2(y)}{|I_P|}\chi_{I_P}\|_{L^\infty(\R)}\lesssim 2^{-(2-\epsilon)m}.
\end{equation}

We put both $E_3$ and $E_3^{*}:=\bigcup_{m}\bigcup_{\T\in \F_m}E_{3,\T}$ in the exceptional set $E$. By \eqref{supportindivtreetwodimdhdyeyehdn}, these have $O(1)$ measure. From this, \eqref{82397eefwjcy9h} and \eqref{saghhgedyoweibnsmaxcb} we have
for each $\T_y\in \F_{m,y}^3$
\be{saghhgedyoweibnsmaxcbdfhjkdshj}
\|\sum_{P\in\T_y}\frac{a_{P_3}^2(y)}{|I_P|}\chi_{I_P}\|_{L^\infty(\R)}\lesssim 2^{-(2-\epsilon)m}.
\end{equation}

A final consequence of Lemma \ref{bessel2dimght86svbns} that we mention is that if $\T\in\F_m$ is a 3-tree, then 
\be{saghhgedyoweibnsmaxcbdfhjkdshjsdhfhkdfhj}
\|(\frac{a_{P_3}^2(y)}{|I_P|})_{P\in\T_y}\|_{\infty}\lesssim 2^{-2m}.
\end{equation}

We will continue to think about $y$ as being fixed.
We have so far learned how to estimate the third component $a_{P_3}(y)$, see \eqref{dhssuayweoiyuicdhsjkalncmkn1jdjk}, \eqref{saghhgedyoweibnsmaxcbdfhjkdshj} and  \eqref{saghhgedyoweibnsmaxcbdfhjkdshjsdhfhkdfhj}.

The control of the first two components $a_{P_i}(y)$, $i\in\{1,2\}$ is completely standard. We will have a purely one dimensional selection algorithm for trees, in particular we will not use two dimensional trees.

\begin{lemma}
Let $i\in\{1,2\}$. Assume that for each $(P,J)\in\P_{hyper}$, $I_P\times J$ intersects the complement of the set
$$E_i:=\{(x,y):M_{p_i}F_i(x,y)\gtrsim 1\}.$$
We also assume as before that $\|F_i\|_{p_i}=1$. Then we can split
$$\P_y=\bigcup_{m\ge 0}\P_{y}^m$$
in such a way that $\P_{y}^m$ is the union of a family $\F_{m,y}^i$ of pairwise disjoint trees satisfying
\be{BMOesthdhd7we7wmd}
\|\sum_{\T_y\in \F_{m,y}^i}\chi_{I_\T}\|_{\BMO_x}\lesssim 2^{2m},
\end{equation}
\be{supportindivtreeonedim}
I_\T\subset \{x:M_{p_3,x}F_3(x,y)\gtrsim 2^{-m}\},
\end{equation}
\be{hhsd6629snmv.jsdt}
\sum_{\T_y\in \F_{m,y}^i}|I_\T|\lesssim 2^{(2+p_3)m}\int M_{p_i,x}^{p_i}F_i(x,y)dx,
\end{equation}
and, if the tree is $j$-tree with $j\not=i$  then 
\be{askhkdluyeuiyoefiuehfuihejqhcjdnknldkjndkl}
\|\sum_{P\in\T_y}a_{P_i,y}^2\frac{\chi_{I_P}}{|I_P|}\|_{\BMO_x}\lesssim 2^{-2m},
\end{equation}
while if the tree is i-tree, then
\be{hsdjkhdjfhuiewyrfuierythcvjdkhjshdjkhaqrwtur}
\|(\frac{a_{P_i,y}^2}{|I_P|})_{P\in\T_y}\|_{\infty}\lesssim 2^{-2m}.
\end{equation}
\end{lemma}

We now put all the pieces together. Put $E_1$ and $E_2$ in $E$.
 Let $\vec{m}:=(m_1,m_2,m_3)$ and assume that each $m_i\ge 0$. Denote by $|\vec{m}|:=m_1+m_2+m_3$. Let $\F_{\vec{m},y}$ the collection of trees obtained by intersecting triples of trees, one from each $\F_{m_i,y}^i$. For such a tree $\T_y$ we get by using \eqref{jdbncklsjhlweufywerifgoq}, \eqref{saghhgedyoweibnsmaxcbdfhjkdshj}, \eqref{saghhgedyoweibnsmaxcbdfhjkdshjsdhfhkdfhj}, \eqref{askhkdluyeuiyoefiuehfuihejqhcjdnknldkjndkl} and \eqref{hsdjkhdjfhuiewyrfuierythcvjdkhjshdjkhaqrwtur}, and by invoking John-Nirenberg again,
\be{saduyedbnjv}
\|\sum_{P\in\T_y}\frac{1}{|I_P|^{1/2}}\prod_{i=1}^{3}a_{P_i}(y)\frac{\chi_{I_P}(x)}{|I_P|}\|_{\infty}\lesssim 2^{-|\vec{m}|(1-\epsilon)},
\end{equation}
To get the above, we actually assume that for each $\T_y$ we have eliminated an exceptional set $E_{\T_y}$ of measure $O(2^{-M|\vec{m}|})|I_{\T_y}|$. More precisely, we add to $E$ the two dimensional exceptional set containing all ${E_{\T_y}}\times \{y\}$, for all $y$. It is easy to see, due to \eqref{hhsd6629snmv.jsdt}, that the union of these sets has measure $O(1)$.

We can now evaluate the BMO norm of the operator associated with the forest $\F_{\vec{m},y}$, defined by 
$$T_{\F_{\vec{m},y}}(x):=\sum_{\T_y\in \F_{\vec{m},y}}\sum_{P\in\T_y}\frac{1}{|I_P|^{1/2}}\prod_{i=1}^{3}a_{P_i}(y)\frac{\chi_{I_P}(x)}{|I_P|}.$$
We have
\begin{align}
\nonumber
\|T_{\F_{\vec{m},y}}\|_{\BMO}&\lesssim 2^{-|\vec{m}|(1-\epsilon)}\|\sum_{\T_y\in \F_{\vec{m},y}}\chi_{I_{\T_y}}\|_{\BMO_x}\\\nonumber&\le2^{-|\vec{m}|(1-\epsilon)}\min_i\|\sum_{\T_y\in \F_{m_i,y}^i}\chi_{I_{\T_y}}\|_{\BMO_x}\\\label{dshcjkhsdkfjk3534774897}&\lesssim 2^{-|\vec{m}|(1-2\epsilon)}2^{\sum_i\frac{2pm_i}{p_i}}
\end{align}
where the last inequality follows from \eqref{dhssuayweoiyuicdhsjkalncmkn1jdjk} and \eqref{BMOesthdhd7we7wmd}.

We note that due to \eqref{supportindivtreetwodim} and \eqref{supportindivtreeonedim}, $T_{\F_{\vec{m},y}}$ is supported in each of the sets
$$\{x:M_{p_i}F_i(x,y)\gtrsim 2^{-m_i}\},$$
It follows that  the size of the support of $T_{\F_{\vec{m},y}}$ is 
$$\lesssim \prod_{i=1}^{3}(2^{p_im_i}\int [M_{p_i}F_i(x,y)]^{p_i}dx)^{\frac{p}{p_i}}.$$
Finally, by invoking this, \eqref{dshcjkhsdkfjk3534774897} and the initial assumption that $\frac{p}{p_i}<2$ we get for sufficiently large $t$ 

\begin{align*}
\|T_{\F_{\vec{m},y}}\|_{t}&\lesssim \|T_{\F_{\vec{m},y}}\|_{\BMO}\prod_{i=1}^{3}(2^{p_im_i}\int [M_{p_i}F_i(x,y)]^{p_i}dx)^{\frac{p}{tp_i}}\\&\lesssim 
2^{\sum_{i}m_i(\frac{p}{t}+\frac{2p}{p_i}-1+\epsilon)}\left(\prod_{i=1}^{3}(\int [M_{p_i}F_i(x,y)]^{p_i}dx)^{\frac{p}{p_i}}\right)^{1/t}\\&\lesssim 2^{-\epsilon|\vec{m}|}\prod_{i=1}^{3}\|M_{p_i}F_i(x,y)\|_{L^{p_i}_x}^{p/t}
\end{align*}

This estimate is summable over all $\vec{m}$ with positive entries. Using this and the fact that 
$$\P_y=\bigcup_{\vec{m}}\bigcup_{\T_y\in \F_{\vec{m},y}}\bigcup_{P\in\T_y}P,$$
we get
$$\int_{E^c} T_y^t(F_1,F_2,F_3)(x)dx\lesssim \prod_{i=1}^{3}\|M_{p_i}F_i(x,y)\|_{L^{p_i}_x}^{p}.$$

Integration in $x$ and H\"older's inequality gives \eqref{nouaecuatiehgstsh}.

\section{The Case 2 and 3}
\label{sec:semidegworse}
We analyze the case
$B=\left[\begin{array}{cc} 0 & 1 \\ 0 & 0\end{array}\right]$, that is
$$\Lambda(F_1,F_2,F_3)=\int_{R^4}F_1(x+t,y+s)F_2(x+s,y)F_3(x,y)K(t,s)dxdydtds.$$
We give an outline of the proof of bounds in the same range as that in Theorem 
\ref{Mainthmforsemidegform}. We first do a cone decomposition as in \eqref{conedecompuuh}, and analyze expressions like

\be{conedecompuuhmoresing}
\sum_{k\in\Z}\int_{R^4} F_1(x+t,y+s)F_2(x+s,y)F_3(x,y)\Psi_k(t)\Phi_k(s)dxdydtds.\end{equation}
As before, we distinguish two cases.

\subsection{The cone $\int\Phi=0$}\ \
\label{sub:1moresingular}

By using standard reductions, in order to get bounds for \eqref{conedecompuuhmoresing} it suffices to prove the boundedness of the model sum
\begin{equation}
\label{model005t5iouivmdfkvnegqwyetefbj}
\int_{\R^2}\sum_{Q=\omega_{1}\times\omega_{2}\times {\omega}_{3}\in {\bf Q}}\prod_{i=1}^3\pi_{\omega_i}^{(i)}F_i(x,y)dxdy
\end{equation}
where for $i\in\{2,3\}$, $\pi_{\omega}^{(i)}$ denotes some projection operator (acting on the $x$ variable) associated with $m_{\omega}$ adapted to $\omega$,
while $\pi_{\omega}^{(1)}$ is the tensor product of a projection as above in the second coordinate and a projection as above on $[-|\omega|,|\omega|]$, in the first coordinate. The relationship in this case between the $\omega_i$ of a given scale $2^k$ is represented by the relations $-\omega_2=\omega_3=\omega_1+C_02^k$. While the fact that $-\omega_2=\omega_3$ is of no particular importance\footnote{$\omega_2=\omega_3$ would have made no difference}, the only genuine source of orthogonality here comes from the fact that $\omega_3=\omega_1+C_02^k$.

We will have two types of trees: the 23-trees, those for which\footnote{we follow here the same notation as in Definition \ref{deftreesg66tgsjhu}} $\bar{\omega}_{\xi_\T,R_\T}\subseteq \bar{\omega}_3$, and 1-trees, those for which $\bar{\omega}_{\xi_\T,R_\T}\subseteq \bar{\omega}_1$. As before, the tree will be called $i$- overlapping if $\xi_\T\in 2\omega_i$  and $i$-lacunary otherwise. The observation that any tree must have at least one lacunary index is exploited to prove the paraproduct estimate (the analog of Proposition \ref{prop:paraproduct11286}). We then work with one and a half dimensional phase-space projections on the second and third function, and with two dimensional projections on the first function. The selection algorithms and the Bessel type inequalities needed to control forests is essentially the same as in Section \ref{sec:34} and Section \ref{sec:345}. As a general observation, we note that, as in the previous case,  neither the fact that $-\omega_2=\omega_3$ nor the separation condition $\omega_3=\omega_1+C_02^k$ play any significant role in the selection algorithm and in establishing Bessel's inequality\footnote{An alternative condition like, say,  $\omega_1=\omega_2=\omega_3$ would have made no difference, in that part of the argument}.

We omit the other details, and invite the interested reader to take this as an exercise, after reading Section \ref{sjckldjksecsecsec}.

\subsection{The cone $\int\Psi=0$}\ \
\label{sub:1moresingularhdfkjdshfjkhsdjfh}


The same standard reductions make it sufficient  to prove the boundedness of the model sum
\begin{equation}
\label{model005sadhhwuqyeuwyddk}
\int_{\R^2}\sum_{Q=\omega_{1}\times\omega_{2}\times {\omega}_{3}\in {\bf Q}}\prod_{i=1}^3\pi_{\omega_i}^{(i)}F_i(x,y)dxdy
\end{equation}
where for $i\in\{2,3\}$, $\pi_{\omega}^{(i)}$ denotes some projection operator (acting on the $x$ variable)  associated with $m_{\omega}$ adapted to $\omega$,
while $\pi_{\omega}^{(1)}$ is the tensor product of a projection as above in the second coordinate and a projection as above on $[|\omega|,2|\omega|]$, in the first coordinate. The relationship in this case between the $\omega_i$ of a given scale $2^k$ is represented by the relations $-\omega_1=\omega_2=-\omega_3+C_02^k$. We will now have 12-trees and 3-trees, and again, there should be at least one lacunary index. Moreover, the fact that $\pi_{\omega}^{(1)}$ projects to  $[|\omega|,2|\omega|]$ in the first coordinate is yet another source of orthogonality. It is easy to see that for either type of tree, the paraproduct estimate follows directly by H\"older's inequality (apply square functions on two of the components, one of which is always $F_1$, and a maximal function on the remaining component). We again omit the details.

\section{The non-degenerate case}
\label{sec:nondeg}
In this section we briefly show how to analyze the case $\{0,1\}\cap Spec(B)=\emptyset.$

For a square $Q$, denote with $c(Q):=(c_1(Q),c_2(Q))$ its center.
Let $\varphi$ be a function whose Fourier transform is adapted to $[-1/2,1/2]\times [-1/2,1/2]$. With each dyadic box $P:=R_P\times\omega_{P_1}\times\omega_{P_2}\times\omega_{P_3}$ with $R_P,\omega_{P_i}\subset\R^2$, and such that $R_P$ has area  $|R_P|$ equal to $|\omega_{P_1}|^{-1}=|\omega_{P_2}|^{-1}=|\omega_{P_3}|^{-1}$, we associate three wave-packets  $\varphi_{P_i}$, localized (in time-frequency) in the tiles $R_P\times \omega_{P_i}$
$$\varphi_{P_i}(x,y)=\frac{1}{|R_P|^{1/2}}\varphi\left(\frac{x-c_1(R_P)}{|R_P|^{1/2}},\frac{y-c_2(R_P)}{|R_P|^{1/2}}\right)e^{i(c_1(\omega_{P_i})x+c_2(\omega_{P_i})y)}.$$

We perform a wave-packet decomposition of each $F_i$ (as in Section \ref{discrtsyshsbsgdtd6}), and a cone decomposition of $K$, and then input these in $\Lambda$.
Elementary computations show that, due to the fact that $\{0,1\}\cap Spec(B)=\emptyset,$ all cones are equivalent. Here is what we mean. These computations show that only a few types of $P$ will produce a non-zero contribution to $\Lambda$. A somewhat simplified way of writing the restrictions on a contributing $P$ with scale $|R_P|=2^{-2k}$ is expressed by the following system of equations:

$$\begin{cases} c_1(\omega_{P_1})+c_1(\omega_{P_2})+c_1(\omega_{P_3})=0 & \\ c_2(\omega_{P_1})+c_2(\omega_{P_2})+c_2(\omega_{P_3})=0 & \\ c_1(\omega_{P_1})+b_{11}c_1(\omega_{P_2})+b_{21}c_2(\omega_{P_2})= C_1 2^{k} & \\ c_2(\omega_{P_1})+b_{12}c_1(\omega_{P_2})+b_{22}c_2(\omega_{P_2})= C_2 2^{k} & \end{cases}$$
with $\max\{C_1,C_2\}>100$. Here $b_{ij}$ are the entries of $B$. It is easy to see that the condition $\{0,1\}\cap Spec(B)=\emptyset$ implies that the family of contributing triples $(\omega_{P_1},\omega_{P_2},\omega_{P_3})$ is one-parameter, in that if we specify $c(\omega_{P_i})$ for some $i$ and specify the scale, then the above system has a unique solution. Moreover, the condition $\max\{C_1,C_2\}>100$ implies that 
$i$- trees will always be $j$- lacunary, for each $j\not=i$. The approach then follows closely the lines of the proof of the boundedness of the one dimensional Bilinear Hilbert Transform, with no significant modifications (see \cite{LT1}, \cite{LT2},\cite{thielelectures} for details). The outcome is bounds for the operator in the same range as that in Theorem \ref{dsjfkhwey237r82347jsdkvnfdbvgu02i0-234i9rejgijrtio}.

\section{Applications to Ergodic theory}
A famous open problem in Ergodic theory concerns the pointwise convergence of the bilinear averages for commuting transformations:
\begin{question}
Let $(X,\Sigma,m)$ be a probability space and let $T,S:X\to X$ be two commuting measurable $m$-preserving point transformations on $X$. Then for each $f,g\in L^{\infty}(X)$, the following averages converge for almost every $x\in X$
\be{hardaverages}
\frac1N\sum_{n=1}^{N}f(T^nx)g(S^{n}x)
\end{equation}
\end{question}

This difficult question is known to have a positive answer when $S$ is a power of $T$. This was proved by Bourgain in  \cite{Bo10}, and then reproved\footnote{In \cite{D}, a unified approach is used to prove convergence of both averages and their singular series counterpart} by the first author in \cite{D}. The techniques we develop in this paper do not seem sufficient by themselves to address the question above in full generality, but we believe they represent an important step towards its resolution. Another step in this program would be to prove bounds for the operator described in Case 6, and ultimately for the bilinear averages
$$\int F_1(x+t,y)F_2(x,y+t)\frac{dt}{t}.$$

We mention however the following related consequences, that come as a by-product of our analysis.
\begin{theorem}
\label{applictoergth}
Under the hypothesis above, the following averages converge for almost every $x\in X$
\be{averages1}
\frac1{N^2}\sum_{n=1}^{N}\sum_{m=1}^{N}f(T^nS^mx)g(T^{-n}S^mx)
\end{equation} 
\be{averages2}
\frac1{N^2}\sum_{n=1}^{N}\sum_{m=1}^{N}f(T^nS^mx)g(T^{m}x)
\end{equation}
\end{theorem}

More generally, we can consider the most general problem of this type, that of the convergence of the averages 
$$\frac1{N^2}\sum_{n=1}^{N}\sum_{m=1}^{N}f(T^{l_1(n,m)}S^{l_2(n,m)}x)g(T^{l_3(n,m)}S^{l_4(n,m)}x),$$
where $l_i$ are linear forms in $n$ and $m$. By doing a case analysis (that we omit) it turns out that all these averages are provable to converge, except for the ones mentioned in the beginning of the section (i.e. $l_1(n,m)=an$, $l_4(n,m)=bn$, $l_2=l_3=0$, and their equivalent versions). This follows either by applying time-frequency methods like in the case of the averages in Theorem \ref{applictoergth}, or by some trivial manipulations that reduce them to more familiar objects. An example of the latter kind is represented by the averages $$\frac1{N^2}\sum_{n=1}^{N}\sum_{m=1}^{N}f(T^nx)g(S^{m}x).$$
While a harmonic analytic approach for them seems unavailable at the moment\footnote{These averages are connected to the singular integral operators described in Case 6}, these averages are easily seen to separate in $n$ and $m$, and their convergence is immediate by the Pointwise Ergodic theorem.

The convergence results in Theorem \ref{applictoergth} are consequences of appropriate oscillation inequalities, as explained below. By using standard transference arguments (see for example \cite{BD}), one can easily show that the convergence is preserved if the probability space $(X,\Sigma,m)$ is replaced with a sigma finite measure space.

The first part of Theorem \ref{applictoergth}  implies the result from \cite{Bo10}, as can easily be seen by choosing $S$ to be the identity transformation. On the other hand, the averages in \eqref{averages2} are of a slightly different nature. While their convergence does not imply the result in  \cite{Bo10}, it nevertheless implies another important result in Ergodic theory, namely the convergence of the Wiener-Wintner averages. In an equivalent formulation, this result asserts the following: Given any dynamical system $(Y,\F,\mu,R)$, any $F\in L^{\infty}(Y)$ and any measurable function $N:Y\to [0,1)$, the averages
$$\frac1N\sum_{n=1}^{N}F(R^ny)e^{iN(y)n}$$
converge for almost every $y\in Y$. See \cite{A3} for an extensive discussion about the Wiener-Wintner property, and \cite{DLTT} and \cite{LTer} for extensions of this result to series. The above implication is easily seen by choosing the sigma finite space to be $X:=Y\times \R$ equipped with the product measure\footnote{$m_{\L}$ denotes the Lebesgue measure} $m:=\mu\times m_{\L}$, then choosing 
$T(y,\theta)=(y,\theta+1)$, $S(y,\theta)=(Ry,\theta)$ and $f(y,\theta)=F(y)$, $g(y,\theta)=e^{iN(y)\theta}$.

We now say a few words about the proof of Theorem \ref{applictoergth}. The argument follows the same lines as that in \cite{D}, with the extra infusion  of techniques developed in this paper. We briefly touch the main points, and leave the details to the interested reader. Let us focus on \eqref{averages1}. Standard transfer between $X$ and $\R^2$ using $\Z^2$ as a mediator\footnote{The transfer from $\R^2$ to $\Z^2$ is  done by using functions constant on all the lattice squares of sidelength 1.  
The transfer from  $\Z^2$ to $X$ is then mediated by functions living on $x$-orbits, that is functions of the form $F(n,m)=f(T^nS^mx)$} shows that it suffices to prove an oscillation inequality for
$$\sum_k\int_{R^2}F_1(x+t,y+s)F_2(x-t,y+s)\Psi_k(t)\Phi_k(s)dtds.$$
We indicate more precisely what this means. Fix an integer $J$ and  a finite sequence of integers $\textbf{U}:=u_1<u_2<...<u_J$.
We restrict attention to the cone in Section \ref{sub:1}, so we will use the notation in there.
\begin{theorem}
\label{jdhcjh239`8e23-4elcm,ndvj89re}
For each $2<p_1,p_2,p_3<\infty$ satisfying $\frac{1}{p_1}+\frac{1}{p_2}+\frac{1}{p_3}=1$, we have
\be{oscillationinec}
\|(\sum_{j=1}^{J-1}\sup_{k\in \Z\atop{u_j\le k<u_{j+1}}}|\sum_{Q=\omega\times\omega\times\bar\omega\in {\bf Q}\atop_{2^k\le |\omega|<2^{u_{j+1}}}}(\pi_{\omega}^{(1)}F_1\pi_{\omega}^{(2)}F_2)*m_{\bar\omega}(x,y)|^2)^{1/2}\|_{p_3'}\lesssim J^{1/4}\|F_1\|_{p_1}\|F_2\|_{p_2},
\end{equation}
where $m_{\bar\omega}$ is a multiplier addapted to $\bar\omega\times [|\bar\omega|,2|\bar\omega|]$.
Moreover, the implicit constant is independent of $J$ and $\textbf{U}$.
\end{theorem}

The important thing in the oscillation inequality above is that the exponent of $J$ is strictly smaller than $1/2$. See \cite{D} for more details.

Consider an arbitrary sequence of functions $h_1,h_2,\ldots,h_{J-1}:\R^2\to\C$ satisfying 
$$\sum_{j=1}^{J-1}|h_j|^2\equiv 1,$$
and also an arbitrary  function $F_3\in L^{p_3}(\R^2)$. We denote by $j(\omega)$ the unique number in $\{1,2,\ldots,J-1\}$ such that $2^{u_{j(\omega)}}\le |\omega|<2^{u_{j(\omega)+1}}$ and by $F_{3,\omega}:=F_3h_{j(\omega)}$. We consider the stopping times $\kappa_j:\R^2\to\{u_j,u_{j}+1,\ldots,u_{j+1}-1\}$, for each $1\le j \le J-1$. Using these, \eqref{oscillationinec} is equivalent to proving that
$$
\sum_{P\in \P}\int_{\R^2} \chi_{R_P,j_P}(x,y)\pi_{\omega}^{(1)}F_1\pi_{\omega}^{(2)}F_2\pi_{\bar\omega}^{(3)}(\chi_{2^{\kappa_{j(\omega)}(\cdot,\cdot)}\le |\omega|<2^{u_{j(\omega)+1}}}F_{3,\omega}(\cdot,\cdot))(x,y)$$
$$\lesssim J^{1/4}\|F_1\|_{p_1}\|F_2\|_{p_2}\|F_3\|_{p_3},$$
where $\pi_{\bar\omega}^{(3)}$ is the projection associated with $m_{\bar\omega}$.
The only difference between this and \eqref{mainineq1444} is the fact that the third function incorporates an extra truncation and an extra block localization. We will have exactly the same kind of trees and sizes for $i\in\{1,2\}$ as in section Section \ref{sub:1}, the only difference being the 3-size, which will have to incorporate these two new ingredients. We define instead the 3-size by
$$\size_3(\T):=\left(\frac{1}{|R_\T|}\sum_{P\in\T}\sup_{m_P}\|\tilde{\chi}_{R_P}^{10}(x,y)T_{m_P}(\chi_{2^{\kappa_{j(\omega_P)}(\cdot,\cdot)}\le |\omega_P|<2^{u_{j(\omega_P)+1}}}F_{3,\omega_P}(\cdot,\cdot))(x,y)\|_{L^2_{x,y}}^2\right)^{1/2},$$
where  $m_P$ ranges over all functions adapted to $\bar\omega_{P}\times [|\bar\omega_{P}|, 2|\bar\omega_{P}|]$.

The phase-space projections in the case $i\in\{1,2\}$, and all the estimates in  Proposition \ref{phase-space} are the same. The only difference is in how we define the phase-space projection of $F_3$. 
We define 
$$\Pi_3(F_3):=\sum_{j\in\J_\T}\tilde{\tilde\chi_j}\tilde\pi_j(\chi_{G_j}F_{3,j})(x,y),$$
where $F_{3,j}=F_{3,\omega}$ and $G_j=\{(x,y): 2^{\kappa_{j(\omega)}(x,y)}\le 2^j\}$ if $|\omega|=2^j$.

We then use 
Proposition \ref{prop:paraproduct11286} and Proposition \ref{phase-space} in the same way as before to get Proposition \ref{mainsingletreeestimate5}. Two things remain to be proved in order to conclude the proof of Theorem \ref{jdhcjh239`8e23-4elcm,ndvj89re}: a bound for $\size_3^{*}$ like the one in Lemma \ref{upperboundestforsize}, and a Bessel type inequality like the one in Proposition \ref{selection}. The first estimate follows by writing
\begin{align*}
\frac{1}{|R_\T|}&\sum_{P\in\T}\|\tilde{\chi}_{R_P}^{10}(x,y)T_{m_P}(\chi_{2^{\kappa_{j(\omega_P)}(\cdot,\cdot)}\le |\omega_P|<2^{u_{j(\omega_P)+1}}}F_{3,\omega_P}(\cdot,\cdot))(x,y)\|_{L^2_{x,y}}^2
\\&\lesssim \frac{1}{|R_\T|}\int \tilde{\chi}_{R_\T}^{8}(x,y)\sum_{j=1}^{J-1}\sum_{u_j\le k<u_{j+1}-1}
\sum_{|\omega_P|=2^k} |T_{m_P}(\chi_{2^{\kappa_{j}(\cdot,\cdot)}\le 2^k} 
F_3(\cdot,\cdot)h_{j}(\cdot,\cdot))(x,y)|^2dxdy\\&\lesssim  \frac{1}{|R_\T|}\int \tilde{\chi}_{R_\T}^{8}(x,y)\sum_{j=1}^{J-1}|F_3(x,y)h_{j}(x,y)|^2dxdy\\&= \frac{1}{|R_\T|}\int  \tilde{\chi}_{R_\T}^{8}(x,y)F_3^2(x,y)dxdy,
\end{align*}
with the penultimate inequality following from the orthogonality of the $T_{m_P}$ for distinct scales, duality and the boundedness of the maximal truncations of two dimensional singular integral operators.

On the other hand, the needed Bessel type inequality was proved in Proposition 5.10. in \cite{D}. That is a one dimensional result, but, as explained before,  the extension to our two dimensional context requires no serious modifications.


\begin{thebibliography}{99}
\bibitem{A3}I. Assani, {\em  Wiener Wintner Dynamical Systems}, 
Erg. Th. \& Dynamical Syst. \textbf{23} (2003), 1637-1654.
\bibitem{BD} E. Berkson and C. Demeter, {\em Spaces of infinite measure and the pointwise convergence of the bilinear Hilbert and ergodic averages defined by $L^p$- isometries}, submitted to the Journal of Operator Theory.
\bibitem{Bo10} J. Bourgain, {\em Double recurrence and almost sure convergence},  J. Reine Angew. Math.  404  (1990), 140-161.
\bibitem{Carleson} L. Carleson, {\em On convergence and growth of partial sums of Fourier series},
Acta Math {\bf 116},   pp 137-157, [1966]
\bibitem{DTT} C. Demeter, T. Tao and C. Thiele,  {\em Maximal multilinear operators}, to appear in TAMS. Available at http://arxiv.org/pdf/math/0510581
\bibitem{DLTT} C. Demeter, M. Lacey, T. Tao and C. Thiele,  {\em Breaking the duality in the return times theorem}, to appear in Duke Math.J. Available at http://arxiv.org/pdf/math/0601455
\bibitem{D} C. Demeter {\em Pointwise convergence of the ergodic bilinear Hilbert transform}, to appear in Ill. Journal of Math. Available at http://arxiv.org/abs/math.CA/0601277
\bibitem{janson} Janson, S., {\em On interpolation of multilinear operators},
in Function spaces and applications (Lund 1986), Lecture Notes in Math.
1302, Springer, Berlin-New York, 1988
\bibitem{LTer} M. Lacey, E. Terwilleger, {\em Wiener-Wintner for Hilbert Transform}, preprint available at http://arxiv.org/abs/math.CA/0601192
\bibitem{LT1} Lacey M. and Thiele C., {\em $L^p$ bounds on the bilinear Hilbert
transform for $2<p<\infty $}, Ann. of Math. {\bf 146}, pp. 693-724, [1997].
\bibitem{LT2} Lacey M. and Thiele C., {\em On Calder\'on's conjecture.},
Ann. of Math. {\bf 149.2}, pp. 475-496, [1999].
\bibitem{LT3} Lacey M. and Thiele C., {\em A proof of boundedness of the Carleson operator},
Math. Res. Letters {\bf 7}, pp. 361-370, [2000]
\bibitem{MTT11} Muscalu C., Pipher, J., Tao T., and Thiele C., {\em Bi-parameter paraproducts},  Acta Math.  {\bf 193}  (2004),  no. 2, 269-296.
\bibitem{MTT8} Muscalu C., Tao T., and Thiele C., {\em Uniform estimates
on multi-linear operators with modulation symmetry}, J. Anal. {\bf 88}, pp. 255-307, [2002].
\bibitem{pramanikterwilleger} Pramanik M. and Terwilleger E., {\em A weak $L^2$ estimate for a maximal
dyadic sum operator on $\R^n$}, Illinois J. Math. {\bf 47}, pp. 775-813, [2003] 
\bibitem{thielelectures} Thiele C., {\em Wave packet analysis}, CBMS {\bf 105}, [2006].
\end{thebibliography}
\end{document}